\documentclass[12pt]{amsart}

\usepackage{amsmath, amsthm, amssymb, bbm, enumerate, verbatim, graphicx}

\allowdisplaybreaks[2]

\usepackage[text={6in,9in},centering]{geometry}
\usepackage[small]{caption}
\usepackage{appendix}

\newcommand{\R}{\mathbb{R}}
\newcommand{\ts}{\hskip+0.5mm}
\newcommand{\N}{\mathbb{N}}
\newcommand{\Z}{\mathbb{Z}}

\newcommand{\SP}{\mathcal{P}}

\newcommand{\BH}{{\bf H}^{\c}_{\lambda}}
\newcommand{\BJ}{{\bf J}}
\newcommand{\BHlim}{{\bf H}^{\vec{c}_\infty}_{\lambda}}
\newcommand{\SH}{\mathcal{H}^{\c}_{\lambda}}
\newcommand{\SJ}{\mathcal{J}}
\newcommand{\SHlim}{\mathcal{H}^{\vec{c}_\infty}_{\lambda}}
\renewcommand{\S}{\mathbb{S}}
\newcommand{\IP}[2]{\left<#1,#2\right>}
\newcommand{\vn}[1]{\lVert#1\rVert}
\newcommand{\rd}[2]{\frac{d#1}{d#2}}
\newcommand{\pd}[2]{\frac{\partial#1}{\partial#2}}

\renewcommand{\c}{\vec{c}_0}
\renewcommand{\k}{\vec{\kappa}}
\newcommand{\ca}{\mathbbm{c}}
\newcommand{\bs}{\star}
\newcommand{\cz}{{\vec c}}
\newcommand{\fg}{{\hat f}}
\newcommand{\df}{\text{\rm d}f\big|_\gamma}

\newcommand{\dkf}[1]{\text{\rm d}^{#1}f\big|_\gamma}
\renewcommand{\L}{\mathbbm{L}}
\renewcommand{\P}{\text{\bf P}}
\newcommand{\PX}{P(\mathcal{X})}
\newcommand{\M}{\vec{M}}
\newcommand{\dca}{\text{\rm d}\mathbbm{c}\big|_{\gamma}}
\let\oldmarginpar\marginpar
\renewcommand\marginpar[1]{\-\oldmarginpar[\raggedleft\footnotesize #1]%
{\raggedright\footnotesize #1}}

\newtheorem{thm}{Theorem}[section]
\newtheorem*{thm*}{Theorem}

\newtheorem{lem}[thm]{Lemma}

\theoremstyle{definition}

\newtheorem{ass}[thm]{Assumption}
\newtheorem{ex}[thm]{Example}
\begin{document}

\title[Global analysis of the generalised Helfrich flow]{Global analysis of the generalised Helfrich flow
of closed curves immersed in $\bf\R^n$}
\author{Glen Wheeler}

\thanks{Financial support from the Alexander-von-Humboldt Stiftung is gratefully acknowledged}
\address{Otto-von-Guericke-Universit\"at\\
Postfach 4120\\
D-39016 Magdeburg}
\email{wheeler@ovgu.de}

\begin{abstract}
In this paper we consider the evolution of regular closed elastic curves $\gamma$
immersed in $\R^n$.
Equipping the ambient Euclidean space with a vector field $\ca:\R^n\rightarrow\R^n$ and a function
$f:\R^n\rightarrow\R$, we assume the energy of $\gamma$ is smallest when the curvature $\k$ of $\gamma$ is
parallel to $\c = (\ca\circ\gamma) + (f\circ\gamma)\tau$, where $\tau$ is the unit vector field spanning the
tangent bundle of $\gamma$.
This leads us to consider a generalisation of the Helfrich functional $\SH$, defined as the sum of the
integral of $|\k-\c|^2$ and $\lambda$-weighted length.
We primarily consider the case where $f:\R^n\rightarrow\R$ is uniformly bounded in $C^\infty(\R^n)$ and
$\ca:\R^n\rightarrow\R^n$ is an affine transformation.
Our first theorem is that the steepest descent $L^2$-gradient flow of $\SH$ with smooth initial data exists for
all time and subconverges to a smooth solution of the Euler-Lagrange equation for a limiting functional
$\SHlim$.
We additionally perform some asymptotic analysis.
In the broad class of gradient flows for which we obtain global existence and subconvergence, there exist many
examples for which full convergence of the flow does not hold.
This may manifest in its simplest form as solutions translating or spiralling off to infinity.
We prove that if either $\ca$ and $f$ are constant, the derivative of $\ca$ is invertible and
non-vanishing, or $(f,\gamma_0)$ satisfy a `properness' condition, then one obtains full convergence
of the flow and uniqueness of the limit.
This last result strengthens a well-known theorem of Kuwert, Sch\"atzle and Dziuk on the elastic
flow of closed curves in $\R^n$ where $f$ is constant and $\ca$ vanishes.
\end{abstract}

\maketitle

\section{Introduction}

%
Consider a closed plane curve immersed via a smooth immersion $\gamma:\S^1\rightarrow\R^2$.
Let us parametrise $\gamma$ by arc-length $s(u) = \int_0^u |\partial_u\gamma|\,du$.
The Helfrich energy, typically used to measure the free energy of a biomembrane \cite{H73}, is given for a
curve by
\begin{equation}
\mathcal{H}^{c_0}_{\lambda}(\gamma)
       = \frac{1}{2}\int_{\gamma} (\kappa-c_0)^2ds + \lambda L(\gamma)\,,
\label{EQhelfrichplanecurve}
\end{equation}
where $\kappa = \IP{\partial_s^2\gamma}{\nu}$ is the curvature of $\gamma$, $\nu$ a unit normal vectorfield on
$\gamma$, $L(\gamma)$ denotes the length of $\gamma$, $\lambda > 0$ is a constant and $c_0:\S^1\rightarrow\R$
is the spontaneous curvature.
If $c_0$ is a constant, then \eqref{EQhelfrichplanecurve} reduces to
\begin{equation}
\mathcal{H}^{c_0}_{\lambda}(\gamma)
       = \frac{1}{2}\int_{\gamma} |\kappa|^2ds + \Big(\lambda+\frac12c_0^2\Big)L(\gamma) + 2c_0\omega\pi\,,
\label{EQhelfrichplanecurve2}
\end{equation}
where $\omega$ is the winding number of $\gamma$.
For curves $\gamma:\S^1\rightarrow\R^n$ with high codimension, the appropriate generalisation of
\eqref{EQhelfrichplanecurve2} is
\begin{equation}
\mathcal{H}^{c_0}_{\lambda}(\gamma)
       = \frac{1}{2}\int_{\gamma} |\k|^2ds + \Big(\lambda+\frac12c_0^2\Big)L(\gamma)\,,
\label{EQhelfrichsimplified}
\end{equation}
where $\k = \partial_s^2\gamma$ is the curvature of $\gamma$.
This functional (in the context of the elastic energy of closed curves) was studied by
Dziuk-Kuwert-Sch\"atzle \cite{DKS02}.  There it was proved:

\begin{thm}[{\cite[Theorem 3.2]{DKS02}}]
For $\lambda, c_0$ constants such that $\lambda + \frac12c_0^2 > 0$ and smooth initial data
$\gamma_0:\S^1\rightarrow\R^n$, the $L^2$-gradient flow $\gamma:\S^1\times[0,T)\rightarrow\R^n$ for
$\mathcal{H}^{c_0}_{\lambda}$ exists for all time.
Furthermore, there exists a sequence of points $p_j\in\R^n$ and times $t_j$ such that the curves
$\gamma(\cdot,t_j) - p_j$ converge, when reparametrised by arc-length, to a smooth solution of the
Euler-Lagrange equation for $\mathcal{H}^{c_0}_{\lambda}$.
\label{TMdks}
\end{thm}

One is naturally led to wonder if a global result similar to Theorem \ref{TMdks} holds for a high-codimension
generalisation of \eqref{EQhelfrichplanecurve} without requiring $c_0$ to be constant, and additionally if the
sequence of translations ${p_j}\in\R^n$ are necessary to demonstrate a limit.
Indeed, the `subconvergence modulo translation' result above does not rule out non-uniqueness of the
limit of the flow or the possibility that the flow floats off to infinity.
Our main results Theorem \ref{TMmt} and Theorem \ref{TMmt2} address each of these issues respectively.

As the curvature of a curve with high codimension is a vector, in order to make sense of the difference
$|\kappa - c_0|$ we replace the spontaneous curvature \emph{function} $c_0$ with a spontaneous curvature
\emph{vector field} $\c$.
In so doing we obtain the \emph{generalised Helfrich functional}:
\begin{equation}
\label{EQhelfrichfunctional}
\SH(\gamma) = \frac12 \int_\gamma |\k - \c|^2ds + \lambda L(\gamma)\,.
\end{equation}
Minimisers of $\SH$ are curves whose curvature vector $\k$ is as parallel as possible to the
spontaneous curvature vector field $\c$.
For this paper, we restrict our attention to spontaneous curvature vector fields $\c$ which are induced by an
ambient vector field $\ca:\R^n\rightarrow\R^n$ and function $f:\R^n\rightarrow\R$ via the immersion $\gamma$;
that is,
\[
\c = \cz + \fg\tau
   = (\ca \circ \gamma) + (f\circ\gamma)\tau,
\]
where $\cz := \ca\circ\gamma$, $\fg := f\circ\gamma$ and $\tau = \partial_s\gamma$ is the unit tangent vector
along $\gamma$.
In particular, we shall not require that $\c$ is a constant vector field.
Note that for $\ca \equiv 0$ and $f(x) = c_0$ we have
\begin{align*}
\mathcal{H}_\lambda^{c_0\tau}(\gamma)
 &= \frac12\int_\gamma \big|\k-c_0\tau\,\big|^2 ds + \lambda L(\gamma)
\\
 &= \frac12\int_\gamma |\k|^2 ds
     + \frac12\int_\gamma c_0^2\, ds
     + \int_\gamma c_0\IP{\k}{\tau}\, ds
     + \lambda L(\gamma)
\\
 &= \frac12\int_\gamma |\k|^2ds
     + \Big(\lambda + \frac12c_0^2\Big) L(\gamma)\,,
\end{align*}
which is \eqref{EQhelfrichplanecurve2} up to an additive constant.
In this paper we are primarily interested in studying the steepest descent $L^2$-gradient flow for $\SH$,
referred to hereafter as the \emph{generalised Helfrich flow}.
This is the one parameter family of immersed curves $\gamma:\S^1\times[0,T)\rightarrow\R^n$ satisfying
\begin{equation}
\label{EQflow}
\pd{}{t}\gamma = -\BH(\gamma)\,,\quad \text{ and }\,\quad 
\rd{}{t}\SH(\gamma) = - \int_\gamma \big| \BH(\gamma) \big|^2ds\,,
\end{equation}
where $\BH(\gamma)$ is the Euler-Lagrange operator for $\SH(\gamma)$ in $L^2$.
Intuitively, the generalised Helfrich flow with initial data $\gamma_0$ is smoothly deforming
$\gamma_0$ in such a way that the curvature vector becomes aligned with $\c$ in as efficient a
manner as possible.
By taking the first variation of $\SH(\gamma)$ in an arbitrary (i.e. not necessarily purely normal)
direction, we prove in Lemma \ref{LMfirstvariation} that $\BH$ is given by
\begin{align*}
\BH(\gamma)
&=
    \nabla^2_s\k
  - \IP{\c}{\tau}\nabla_s\k
  - \nabla_s^2\cz
  + \frac12|\c-\k|^2\k
  - \big(\lambda+|\c|^2\big)\k
  - \IP{\dca(\tau)}{\tau}\k
\\&\quad
  - \Big[\big(\dca\big)^T(\k-\c)
   + \df(\tau)\,\cz
   + \fg\dca(\tau)
   - \IP{\c}{\tau}\big(\df\big)^T \Big]^\bot
\\&\quad
  - \df(\tau)\,\k
  + \fg\IP{\cz}{\tau}\k\,.
\end{align*}
In the above we have used $\nabla_s$ to denote the normal derivative (the projection of $\partial_s$ onto the
normal bundle) and $\df$, $\dca$ to denote the derivative of $f$ and $\ca$ evaluated at $\gamma$.
The standard Euclidean inner product is denoted by $\IP{\cdot}{\cdot}$.
The symbols $(\cdot)^T$ and $(\cdot)^\perp$ denote transposition and normal projection respectively, so that
$\Big[\big(\dca\big)^T(\k-\c)\Big]^\bot$ and $\Big[\big(\df\big)^T\Big]^\bot$ are vectors.
We refer the reader to Section 2 for further exposition on our notation.

As one easily verifies, the Euler-Lagrange operator $\BH$ applied to $\gamma$ yields a highly non-linear
fourth order parabolic system of tightly coupled differential equations.
With the notable exception of \cite{DKS02}, earlier works on non-linear fourth order flows of curves, such as
the curve diffusion flow (see \cite{EG97,P96,W12}), the elastic flow or curve straightening flow (see
\cite{DKS02,K96,LS87,LS85,P96,W93,W95}), and the affine curve lengthening flow (see \cite{A99}), have been
typically carried out in in the context of plane curves.
There the normal bundle of $\gamma$ is trivial and one studies a single highly non-linear fourth order
equation.
In our situation however, the Euler-Lagrange operator $\BH$ is a vector in the normal bundle, which may have
quite complicated geometry.
This, and the presence of a non-constant spontaneous curvature $\c$, complicates much of the analysis.

Closed particular solutions of $\BH(\gamma) = 0$ are somewhat difficult to grasp in any generality.
Understanding the case of simple circles $S_\rho:\S^1 \rightarrow \R^n$, defined for some choice of
coordinates by $S_\rho(s) = \rho(\cos \frac{s}{\rho}, \sin\frac{s}{\rho}, 0, \ldots, 0)$, is already quite
complicated.
Some elementary computations yield the following.

\begin{lem}
Let $S_\rho:\S^1\rightarrow\R^n$ be a simple circle with radius $\rho$.
Suppose $f:\R^n\rightarrow\R$ is a constant and $\lambda > 0$.
If $\ca$ is a translation, then either
\begin{enumerate}
\item[(i)] $\ca$ is zero and $S_\rho$ is critical for $\SH$ if and only if 
$\frac1\rho = \sqrt{2\lambda + |f|^2}$; or
\item[(ii)] $\ca$ is non-zero and $S_\rho$ is never critical for $\SH$.
\end{enumerate}
If $\ca$ is a rotation through an angle $\theta$ satisfying $\ca(S_\rho(\S^1)) = S_\rho(\S^1)$, then either
\begin{enumerate}
\item[(iii)] $\ca(S_\rho(s)) = S_\rho(s)$ for every $s\in\S^1$ and $S_\rho$ is critical for $\SH$ if
and only if 
$\rho = \sqrt{\frac{\sqrt{(2 + 2\lambda + |f|^2)^2 +12}-(2 + 2\lambda + |f|^2)}{6}}$
; or
\item[(iv)] $\ca$ is a rotation of the plane spanned by the axes of $S_\rho$, and $S_\rho$ is critical only if $\rho$
is a real positive root of the polynomial
\[Q(\rho) = \rho^{-4} - (
     2\cos\theta
   + 2\lambda
   + |f|^2
    )\rho^{-2}
  - (4f\sin\theta)\rho^{-1}
  - 3\,.
\]
\end{enumerate}
\label{LMcirclesclassification}
\end{lem}

The proof of Lemma \ref{LMcirclesclassification} is presented in the appendix.

Let us briefly discuss two aspects of Lemma \ref{LMcirclesclassification}.
Statement $(ii)$ gives the non-existence of circular critical curves when $\ca$ is a translation and
$f$ is a constant.
Indeed, classifying all closed solutions of the Euler-Lagrange equation for $\SH$ appears to be difficult even
in this very special case.
Their existence (and smoothness), however, is a corollary of Theorem \ref{TMmt}.

We find it curious that the angle of rotation plays a relatively important role in the size of a critical
circle, as one readily observes from statement $(iv)$.
Unfortunately, it is quite clumsy to write down the roots of $Q$ explicitly, although one
should note that in the case where the rotation is trivial ($\theta = 2k\pi$ for $k\in\Z$) the only root is
given by the expression in $(iii)$.
In order to demonstrate the impact of the rotation angle on the radius of the critical circle, we have 
in Figure \ref{FG1} set $\lambda=f=1$ (so that one obtains only a single allowable radius for each angle of
rotation) and graphed the resultant radii against each angle in Figure \ref{FG1}.

\begin{figure}[t]
\centering
\includegraphics[scale=1]{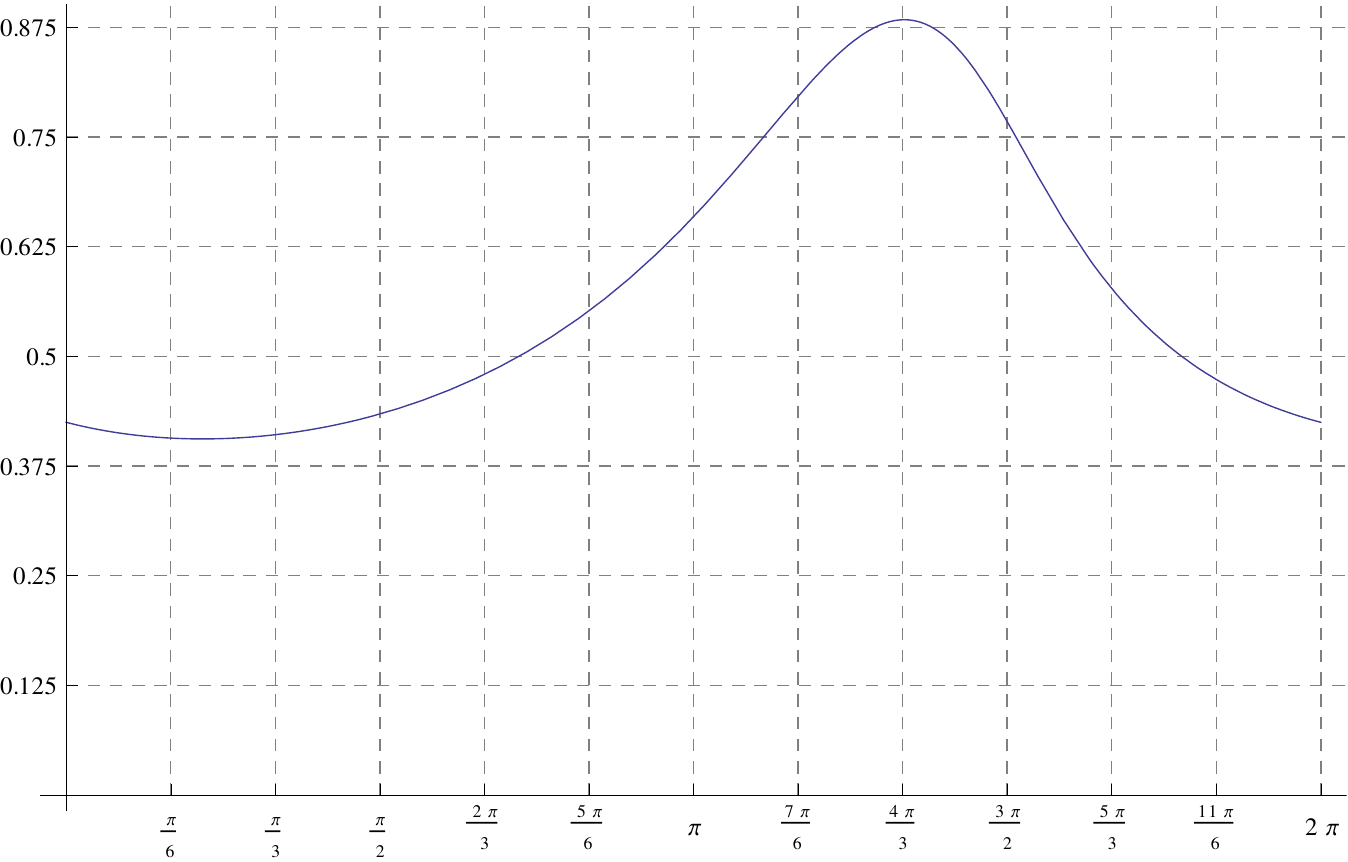}
\caption{
Graph of the allowable radii $\rho$ (vertical) of the critical circle $S_\rho$ against the angle $\theta$
(horizontal) of the rotation $\ca$ in $(iv)$ of Lemma \ref{LMcirclesclassification} when $f=\lambda=1$.
}
\label{FG1}
\end{figure} 

Local existence for the flow \eqref{EQflow} is a standard matter; one may follow the approach of \cite{DKS02},
noting that the principal part of $\BH$ is quasilinear and applying standard theory.
Other possible approaches include generalising \cite{MM11} to high codimension through the estimates of
\cite{L01}, or recasting the problem in a form such that the general existence theory of
\cite{B11} applies.

\begin{thm}[Local existence]
\label{STE}
Suppose $\gamma_0:\S^1\rightarrow\R^n$ is a smooth regular curve and $\c = (\ca\circ\gamma) + (f\circ\gamma)\tau$
with $\ca\in C^\infty(\R^n,\R^n)$, $f\in C^\infty(\R^n)$. 
Then there exists a $T\in(0,\infty]$ and a unique one-parameter family of immersions
$\gamma:\S^1\times[0,T)\rightarrow\R^n$ such that
\begin{enumerate}
\item[(i)]
$\gamma(0,\cdot) = \gamma_0$;
\item[(ii)]
$\partial_t\gamma = -\BH(\gamma)$;
\item[(iii)]
$\gamma(\cdot,t)$ is of class $C^\infty$ and periodic of period $L(\gamma(\cdot,t))$ for every $t\in(0,T)$;
\item[(iv)]
$T$ is maximal.
\end{enumerate}
\end{thm}

We shall work in the class of generalised Helfrich flows where $f$ and $\ca$ satisfy the following
pair of assumptions.

\begin{ass}
\label{Aca}
The ambient vector field
$\ca:\R^n\rightarrow\R^n$ is an affine transformation,
\begin{equation*}
\ca(x) = \L x + \M\,,\quad\text{for }
x\in\R^n\,,
\end{equation*}
where $\L$ is a constant $n\times n$ matrix with the property that
\begin{equation*}
\L\text{ is positive semi-definite}\qquad\text{or}\qquad |\L| \le \lambda,
\end{equation*}
and $\M$ is a constant vector in $\R^n$.
For any matrix $\mathbbm{A}$ we use $|\mathbbm{A}|$ to denote the (induced) operator norm of $\mathbbm{A}$.
\end{ass}

Let us briefly remark that this includes the simpler cases of constant ambient curvature, locally isometric
ambient curvature (where $\ca$ is Killing on $\R^n$) and linear ambient curvature.

\begin{ass}
\label{Af}
The ambient function $f$ is of class $C^\infty$ and satisfies bounds
\begin{equation*}
\sup_{x\in \R^n} \Big|\text{d}^m f\big|_x\Big| \le c_{m}\,,
\end{equation*}
with $m\in\N_0$ and $c_{m}$ a constant depending only on $m$.
\end{ass}

Our first theorem is that under Assumptions \ref{Aca} and \ref{Af} the generalised Helfrich flow with smooth
initial data exists for all time, never encountering a singularity.
We also show that the flow converges up to translation to a solution of the Euler-Lagrange equation for a
limiting functional $\mathcal{H}^{\vec{c}_\infty}_\lambda$, where $\vec{c}_\infty$ is the limit of the
spontaneous curvature along an appropriately chosen subsequence.

\begin{thm}
\label{TMmt}
Let $\gamma:\S^1\times[0,T)\rightarrow\R^n$ be a generalised Helfrich flow.
Suppose $\ca$ and $f$ fulfil Assumptions \ref{Aca} and \ref{Af} respectively and $\lambda>0$.
Then $T=\infty$.
Furthermore, there exists a sequence of times $t_j\rightarrow\infty$, and a sequence of points $p_j\in\R^n$
such that the curves $\gamma(\cdot,t_j)-p_j$ converge, when reparametrised by arclength, to a smooth
curve $\gamma_\infty$.
Along this sequence of times the spontaneous curvature $\c(s,t_j) = \big((\ca\circ\gamma)(s,t_j) +
(f\circ\gamma)(s,t_j)\tau(s)\big)$ also converges smoothly with limit
$\vec{c}_\infty(s) = \lim_{j\rightarrow\infty} \big((\ca\circ\gamma)(s,t_j) +
(f\circ\gamma)(s,t_j)\tau(s)\big)$.
The curve $\gamma_\infty$ is critical for the functional $\SHlim$ and satisfies
\begin{equation}
\BHlim(\gamma_\infty)
=0\,,
\label{EQeulerlagrange}
\end{equation}
where $\BHlim$ is the Euler-Lagrange operator for $\SHlim$.
\end{thm}

We remark that there is no smallness condition required for this theorem: the initial data may possess
self-intersections as well as arbitrarily high and concentrated initial energy.

It is not possible to strengthen the subconvergence modulo translation result of Theorem \ref{TMmt}
without imposing additional restrictions on $f$ or $\ca$.
Indeed, solutions satisfying the conditions of Theorem \ref{TMmt} may exhibit rather wild asymptotic behaviour.
In general, one may not strengthen the subconvergence statement given above.
The flow may translate off to infinity or spiral outwards to infinity for example, and the extracted limiting
curve $\gamma_\infty$ which is critical for the limiting functional $\mathcal{H}^{\vec{c}_\infty}_\lambda$
need not be unique.
We demonstate the first of these behaviours with a simple example.

\begin{ex}
\label{EXtranslatingcircle}
Let $\ca \equiv 0$ and $f(x) = \frac{1}{1+|x|^2}$.
Then
\begin{align}
\notag
-\BH(\gamma)
&=
  - \nabla^2_s\k
  + \fg\nabla_s\k
  - \frac12|\fg\tau-\k|^2\k
  + \big(\lambda+|\fg|^2\big)\k
\notag\\&\qquad
  - \Big[
    \IP{\c}{\tau}\big(\df\big)^T
    \Big]^\bot
  + \df(\tau)\,\k
\notag\\
 &= -\nabla_s^2\k + \fg\nabla_s\k - \frac12|\k|^2\k + \lambda\k + \frac12|\fg|^2\k
\notag\\&\qquad
  - \fg\bigg( \frac{-2\gamma}{(1+|\gamma|^2)^2} - \IP{\frac{-2\gamma}{(1+|\gamma|^2)^2}}{\tau}\tau
       \bigg)
  + \IP{\frac{-2\gamma}{(1+|\gamma|^2)^2}}{\tau}\k\,.
\label{EQelfordecayingf}
\end{align}
Let us first consider the case where $\gamma_0 = S_\rho = \rho(\cos \frac{s}{\rho},\sin \frac{s}{\rho}, 0,
\ldots, 0)$ is a circle with radius $\rho$ centred at the origin.
Such circles have the special property that
$\k = -\rho^{-2}\gamma$.
From \eqref{EQelfordecayingf} we compute
\begin{align}
\notag
-\BH(S_\rho)
 &= -\frac{\tau}{1+|\gamma|^2}\bigg(\rho^{-2} - |\k|^2\bigg)
 - \k\bigg(
        \rho^{-4} + |\k|^2 + \frac12|\k|^2 - \lambda - \frac1{2(1+|\gamma|^2)^2}
\\
\notag
&\qquad\qquad
      + \frac2{(1+|\gamma|^2)^2}\IP{\gamma}{\tau}
      + \frac2{(1+|\gamma|^2)^3}\rho^2
      \bigg)
\\
\notag
 &= 
  S_\rho\bigg(
        \rho^{-6} + \frac32\rho^{-4} - \lambda\rho^{-2} - \frac1{2\rho^2(1+|\rho|^2)^2}
      + \frac2{(1+|\rho|^2)^3}
      \bigg)
\end{align}
The solution with initial data $S_\rho$ therefore flows purely by homethety.  Furthermore, if $\rho$
is small, then $\rho^{-6}$ dominates the equation, and causes the circle to expand.  If $\rho$ is
large, then $-\lambda\rho^{-2}$ dominates, and causes the circle to shrink.  There is a
$\lambda_\rho = \lambda_\rho(\rho,\lambda)$ to which the flow converges.

Let us now take $\tilde{S}_\rho = S_\rho + \rho e_3 = \rho(\cos \frac{s}{\rho},\sin \frac{s}{\rho}, 1, 0,
\ldots, 0)$ as initial data.
Note that for this initial data we have $\k = -\rho^{-2}\tilde{S}_\rho + \rho^{-2}e_3$.
Using \eqref{EQelfordecayingf} we compute
\begin{align}
\notag
-\BH(\tilde{S}_\rho)
 &= -\nabla_s^2\k + \fg\nabla_s\k - \frac12|\k|^2\k + \lambda\k + \frac12|\fg|^2\k
\notag\\&\qquad
  - \fg\bigg( \frac{-2\tilde{S}_\rho}{(1+|\tilde{S}_\rho|^2)^2} -
    \IP{\frac{-2\tilde{S}_\rho}{(1+|\tilde{S}_\rho|^2)^2}}{\tau}\tau
       \bigg)
  + \IP{\frac{-2\tilde{S}_\rho}{(1+|\tilde{S}_\rho|^2)^2}}{\tau}\k\,.
\notag
\\
 &= \k\bigg(
     - \frac12\rho^{-2}
     + \lambda 
     + \frac{1}{2(1+2|\rho|^2)^2}
      \bigg)
     + (e_3-\rho^2\k) \bigg(
       \frac{1}{(1+2|\rho|^2)^3}
      \bigg)\,,
\notag
\intertext{setting $\rho = 1/\sqrt{2}$ and $\lambda = \frac{15}{16}$ we have}
-\BH(\tilde{S}_\rho)
 &= \k\bigg(
       \lambda 
     - 1
     + \frac18
     - \frac1{16}
      \bigg)
     + \frac18e_3
  = 
     \frac18e_3\,.
\end{align}
The solution is in this case therefore given by $\gamma(s,t) = S_{1/\sqrt{2}}(s) + \frac{t}{8}e_3$
and slides off to infinity.  Although composing the flow with translations will allow one to extract
a convergent sequence of curves (which are critical for a limiting functional), the flow itself does
not converge.
\end{ex}


Our second theorem provides sufficient conditions under which one may prevent this lack of compactness for the
flow.
It identifies three somewhat independent cases where we may recover full convergence of the flow and
uniqueness of the limit $\gamma_\infty$.
One is a `properness'-type property of the pair $(f,\gamma_0)$.

\begin{ass}
\label{ASfproperlike}
Let $\gamma_0:\S^1\rightarrow\R^n$ be the initial curve in a one-parameter family of $\SH(\gamma_0)$-bounded
curves.
There exists an $R \in [0,\infty)$ such that for all $x$ with $R \le |x| \le
R+(2\lambda)^{-1}\SH(\gamma_0)$,
\begin{equation}
\label{EQconditionfproperlike}
|f(x)| > \frac{\sqrt{3}\SH(\gamma_0)}{\pi}\,.
\end{equation}
\end{ass}

Families of $\SH(\gamma_0)$-bounded curves are discussed in Section 3, and include any generalised Helfrich
flow with initial data $\gamma_0$.

We are also able to obtain full convergence of the flow if $\dca$ is invertible (and non-vanishing),
or if $\ca$ and $f$ are constant.

\begin{thm}
\label{TMmt2}
Let $\gamma:\S^1\times[0,T)\rightarrow\R^n$ be a generalised Helfrich flow satisfying the conditions of
Theorem \ref{TMmt}.
Suppose that one of the following conditions holds:
\begin{enumerate}
\item[(i)]
$\ca$ and $f$ are constant;
\item[(ii)]
$\text{d}\mathbbm{c}$ is invertible and non-vanishing;
\item[(iii)]
The pair $(f,\gamma_0)$ satisfy Assumption \ref{ASfproperlike}.
\end{enumerate}
Then the flow $\gamma$ converges to a unique limit $\gamma_\infty$ which is a critical point
of the Euler-Lagrange equation \eqref{EQeulerlagrange}. 
\end{thm}

We remark that Theorem \ref{TMmt2}, case $(i)$, strengthens the convergence statement of \cite[Theorem
3.2]{DKS02}, where $\ca \equiv 0$ and $f = c_0 \in \R$.

The proof of Theorem \ref{TMmt} follows by obtaining uniform a-priori estimates for all normal derivatives of
the curvature in $L^2$, and then converting these into bounds on all Sobolev norms of $\gamma$ in the initial
parametrisation, a standard strategy which was employed in \cite[Proof of Theorem 3.2]{DKS02} for example.


Proving these estimates directly, as is done in \cite{DKS02} for the curve straightening flow in $\R^n$,
appears to be quite difficult.
One encounters a large number of extra terms related to $\c$, and since $\c$ is induced by $\gamma$ these are
interwoven with derivatives of $\gamma$.
A natural strategy is to interpolate these terms away; unfortunately, one faces difficulty in doing this as the
normal derivatives $\nabla_s$ become distorted by $\dca = \L$.
This pair of operations do \emph{not} commute:
\begin{equation*}
\nabla_s \L X - \L\nabla_sX = \IP{\tau}{\partial_sX}\L\tau - \IP{\tau}{\L\partial_sX}\tau
 \ne 0\,.
\end{equation*}
The failure of commutation is in general of the same order as the operations themselves, and so a standard
bootstrapping approach seems unlikely to succeed.
There is a crucial observation to be made however, which is the key idea behind the estimates of Section 4: If
$X$ is a \emph{normal} vector field, then the \emph{normal projection} of the commutator of $\nabla_s$
and $\L$ is
\begin{equation*}
\Big[\nabla_s \L X - \L\nabla_sX\Big]^\perp
 = -\IP{\k}{X}\big[\L\tau\big]^\perp\,, 
\end{equation*}
which is controlled by $\vn{X}_{L^\infty}$ and $\vn{\k}_{L^\infty}$.
By constraining ourselves to obtaining estimates where we must only control the normal projection of the
commutator of $\nabla_s$ and $\L$ of a normal vector field $X$, bootstrapping becomes possible.

Clearly, regardless of the order of $X$, one must first obtain uniform a-priori estimates for $\k$ in
$L^\infty$ for this strategy to have a chance of success.
We produce these estimates through bounding $\nabla_s\k$ in $L^2$.
For this we in turn require a number of preliminary results.
We prove a-priori estimates for the elastic energy of the curve, the length of the curve, the
norm of $\cz$, and in some special cases the length of the position vector.
These follow from the uniform boundedness of the energy $\SH(\gamma_0)$ and are not particular to the flow
\eqref{EQflow}.
These simple estimates are then combined with an interpolation argument for the evolution of
$\vn{\nabla_s\k}_{L^2}^2$ ($\nabla_t\nabla_s\k$ is a normal vector field, and when projected against
$\nabla_s\k$ only the normal component of the commutator need be estimated) which gives a-priori control of
$\vn{\k}_{L^\infty}$.
This allows us to enact a relatively simple argument to obtain a-priori control of
$\vn{\nabla^2_s\k}_{L^2}^2$, which is enough to allow us to carry out our more complicated argument to finally
obtain uniform estimates for $\vn{\nabla^m_s\k}_{L^2}^2$.

The assumptions on $f$ and the uniform bounds on all derivatives of curvature then imply global
existence and subconvergence modulo translation, which is Theorem \ref{TMmt}.
For the finer results of Theorem \ref{TMmt2}, we require two arguments.
For flows satisfying either condition $(ii)$ or condition $(iii)$ of Theorem \ref{TMmt2}, we find a great ball
$B_\rho(0)\subset\R^n$, to which we can contain the image of the flow $\gamma$.
This allows us to show directly the uniqueness of any limit given by Theorem \ref{TMmt} as well as remove the
translations.
Flows satisfying condition $(i)$ are distinguised by their corresponding functionals $\SH$ being
translation-invariant.
It thus seems difficult to imprison the flow a-priori in a great ball.
It is nevertheless possible to obtain a Lojasiewicz-Simon inequality (using an 
argument completely analogous to one contained in Simon's fundamental work \cite{S83}) for the
functional $\SH$ in a neighbourhood of a limit $\gamma_\infty$ given by Theorem \ref{TMmt}.
The translation invariance of the functional can then be exploited to obtain this inequality in any
translated neighbourhood.
If one attempts to proceed as in \cite{S83}, difficulty is encountered due to the time dependence of
the measure $ds$.
Following an idea of Andrews \cite{A97}, we write the flow for sufficiently large time as a
graph over the candidate limit $\gamma_\infty$.  Fixing the measure $ds_0$ (to which $ds$ is
equivalent), we consider an associated functional $\SJ$, which is essentially $\SH$ with the
evolving measure replaced by $ds_0$.
Although the flow is not the gradient flow of $\SJ$, the angle between the Euler-Lagrange operator
$\BJ$ and $\BH$ is bounded for sufficiently large time away from $\pi/2$.  This is enough to then
enact the argument of Simon with rspect to the functional $\SJ$, and establishes the desired
convergence result in this remaining case.

This paper is organised as follows.
In Section 2 we set our notation and compute the first variation of $\SH$.
Section 3 is concerned with deriving some basic consequences of the conservation of the energy.  The results
of this section (apart from Lemma \ref{LMaprioricontrolofGammaf}, which additionally requires the
family to be continuous) hold for any family of immersed curves with uniformly bounded $\SH$.
In Section 4 we obtain a-priori estimates for all derivatives of curvature and prove Theorems \ref{TMmt} and
\ref{TMmt2}.
The Appendix contains the proof of Lemma \ref{LMcirclesclassification}.


\section*{Acknowledgements}

The author would like to thank Ben Andrews for enlightening discussions related to Theorem \ref{TMmt2},
which directly led to the proof of Theorem \ref{TMmt2}, part $(i)$.
This work was completed with the financial support of the Alexander-von-Humboldt Stiftung at the
Otto-von-Guericke Universit\"at Magdeburg, which the author gratefully acknowledges.

\section{Notation and First Variation}

Suppose $\gamma:\S^1\rightarrow\R^n$, $n\ge2$, is a regular smooth immersed closed curve.
The length of $\gamma$ is
\[
L(\gamma) = \int_{\S^1} |\partial_u\gamma|\, du.
\]
We shall keep $\gamma$ parametrised by arc-length $s$, where $ds = |\partial_u\gamma|\, du$.
In this case for notational brevity we identify the parameter domain $\S^1$ with the interval $[0,L(\gamma))$
with its endpoints identified.
Integrals over $\gamma$ are to be interpreted as integrals over the interval of periodicity.

The fundamental geometric objects associated with $\gamma$ are its unit tangent vector field $\tau :=
\partial_s\gamma$ and its curvature $\k := \partial^2_s\gamma$.
Apart from the partial derivative $\partial$, we shall also use the projection onto the normal bundle over
$\gamma$ of $\partial$, denoted by $\nabla$ and defined 
for a vector field $X:\S^1\rightarrow\R^n$ by
\[
\nabla X = \partial X - \IP{\partial X}{\tau}\tau\, ,
\]
where we have used $\IP{\cdot}{\cdot}$ to denote the standard inner product on $\R^n$.
Introducing the normal projection $[\cdot]^\bot$, defined by $[X]^\bot = X - \IP{X}{\tau}\tau$, we write this
succinctly as
\[
\nabla X = \big[\partial X\big]^\bot\,.
\]
Supposing now that $X$ and $Y$ are normal vector fields along $\gamma$, we have $\partial_s\IP{X}{Y} =
\IP{\nabla_sX}{Y} + \IP{X}{\nabla_sY}$ and so we observe the following integration by parts formula
\begin{equation*}
\int_\gamma\IP{\nabla_sX}{Y}\,ds = -\int_\gamma\IP{X}{\nabla_sY}\,ds\,.
\end{equation*}
Clearly, if either of the vector fields $X$, $Y$, are not completely normal, then one must first pass to the
partial derivative $\partial_s$ before performing integration by parts.

We use $X^T$ to denote the transpose of $X$.
We use this notation quite often in the context of the identity
\begin{equation*}
\IP{X}{MY} = \IP{M^TX}{Y}\,
\end{equation*}
where $M$ is an $(n\times n)$ matrix and $X$, $Y$ vectors in $\R^n$.

We shall extend the $P$-style notation used in \cite{DKS02}.
Let $\SP^m$ denote the set of all permutations of $\{1,\ldots,m\}$.
Given vector fields $X_i$, $i=1,\ldots,k$, $k\in\N$, and a permutation $i\in\SP^k$, we denote by
$X_1*\cdots*X_k$ a term of the form
\begin{equation}
X_1*\cdots*X_k
=
\begin{cases}
\IP{X_{i_1}}{X_{i_2}}\cdots\IP{X_{i_{k-1}} }{X_{i_k}}\,,&\text{ for $k$ even}
\\
\IP{X_{i_1}}{X_{i_2}}\cdots\IP{X_{i_{k-2}} }{X_{i_{k-1}} }X_{i_k}\,,&\text{ for $k$ odd}\,.
\end{cases}
\label{EQstarnotation}
\end{equation}
As in \cite{DKS02}, we also allow that some of the $X_i$ are functions, in which case the $*$-product reduces
to multiplication.
We further extend the $*$ product to act upon $1$-forms $\omega_i$, by setting ($X_i$ vector fields,
$\omega_i$ 1-forms)
\begin{align*}
\IP{\omega_i}{X_j} &:= \IP{X_j}{\omega_i} := \omega_i(X_j) = \IP{\omega_i^T}{X_j}\,,\quad\text{and}
\\
\IP{\omega_i}{\omega_j} &:= \omega_i(\omega_j^T) = \omega_j(\omega_i^T) = \IP{\omega_i^T}{\omega_j^T}\,.
\end{align*}
We use the notation $P_\nu^{\mu,k}(X)$ to denote any linear combination of terms of the type
$\nabla_s^{i_1}X*\cdots*\nabla_s^{i_\nu}X$ with universal, constant coefficients, where $\mu = i_1 + \cdots +
i_\nu$ is the total number of derivatives, each of maximum order $k$; that is, $i_j \le k$ for $j = 1,\ldots,\nu$.
We observe the properties
\begin{align*}
P_\nu^{\mu,k}(X)
 &= P_\nu^{\mu,k+l}(X),\text{ for any $l\ge0$}\,,
\\
P_\nu^{\mu,k}(X) * P_\beta^{\alpha,l}(X)
 &= P_{\nu+\beta}^{\mu+\alpha,\,\max\{k,l\}}(X)\,,
\\
\int_\gamma \nabla_sX*P_\nu^{\mu,k}(Y)\,ds
 &= -\int_\gamma X*P_\nu^{\mu+1,k+1}(Y)\,ds\,,\text{ for $X,Y$
normal.}
\end{align*}
Throughout the paper we use the abbreviations
\begin{align*}
P_\nu^{\mu,k}(X) + P_\beta^{\alpha,l}(X)
  = (P_\nu^{\mu,k} + P_\beta^{\alpha,l})(X)\,,\quad\text{and}\quad
P_\nu^\mu(X)
  = P_\nu^{\mu,\mu}(X)\,.
\end{align*}
We shall further use the notation $P_{\nu_1,\nu_2}^{\mu,k}(X;Y)$ to denote any linear combination of terms of
the form
\begin{equation*}
\Big(\nabla_s^{i_1}X*\cdots*\nabla_s^{i_{\nu_1}}X\Big)
*
\Big(\nabla_s^{i_{\nu_1+1}}Y*\cdots*\nabla_s^{i_{\nu_1+\nu_2}}Y\Big)
\end{equation*}
with universal, constant coefficients, where $\mu = i_1 + \cdots + i_{\nu_1+\nu_2}$ is the total number of
derivatives, each of maximum order $k$; that is, $0 \le i_j \le k$ for $j = 1,\ldots,\nu_1+\nu_2$.

We now compute the Euler-Lagrange operator and steepest descent $L^2$-gradient flow of $\SH$.

\begin{lem}
\label{LMfirstvariation}
Suppose $\gamma:\S^1\rightarrow\R^n$ is a smooth, closed, immersed curve.  Then the first variation of $\SH$ at
$\gamma$ in the direction of a vector field $\phi:\S^1\rightarrow\R^n$ (not necessarily normal) is given by
\begin{align*}
\rd{}{t}&\SH(\gamma + t\phi)\bigg|_{t=0}
=
- \int_\gamma \bigg\langle \phi,
  - \nabla^2_s\k
  + \IP{\c}{\tau}\nabla_s\k
  + \nabla_s^2\cz
  - \frac12|\c-\k|^2\k
\\*
&\hskip+2.5cm
  + \Big[\big(\dca\big)^T(\k-\c)
   + \df(\tau)\,\cz
   + \fg\dca(\tau)
   - \IP{\c}{\tau}\big(\df\big)^T \Big]^\bot
\\*
&\hskip+2.5cm
  + \big(\lambda+|\c|^2\big)\k
  + \IP{\dca(\tau)}{\tau}\k
  + \df(\tau)\,\k
  - \fg\IP{\cz}{\tau}\k
                       \bigg\rangle\,ds\,.
\end{align*}
Critical points $\gamma$ of $\SH$ satisfy $\BH(\gamma) \equiv 0$, where $\BH$ is the Euler-Lagrange
operator for $\SH$ in $L^2$ given by
\begin{align*}
\BH(\gamma)
&=
  \nabla^2_s\k
  - \IP{\c}{\tau}\nabla_s\k
  - \nabla_s^2\cz
  + \frac12|\c-\k|^2\k
  - \big(\lambda+|\c|^2\big)\k
  - \IP{\dca(\tau)}{\tau}\k
\\&\quad
  - \Big[\big(\dca\big)^T(\k-\c)
   + \df(\tau)\,\cz
   + \fg\dca(\tau)
   - \IP{\c}{\tau}\big(\df\big)^T \Big]^\bot
\\&\quad
  - \df(\tau)\,\k
  + \fg\IP{\cz}{\tau}\k\,.
\end{align*}
The steepest descent $L^2$-gradient flow of $\SH$ with initial data $\gamma_0$ is the one-parameter family of
immersions $\gamma:\S^1\times[0,T)\rightarrow\R^n$ satisfying $\gamma(\cdot,0) = \gamma_0$ and
\begin{align*}
\partial_t\gamma
&=
  - \nabla^2_s\k
  + \IP{\c}{\tau}\nabla_s\k
  + \nabla_s^2\cz
  - \frac12|\c-\k|^2\k
  + \big(\lambda+|\c|^2\big)\k
  + \IP{\dca(\tau)}{\tau}\k
\\&\quad
  + \Big[\big(\dca\big)^T(\k-\c)
   + \df(\tau)\,\cz
   + \fg\dca(\tau)
   - \IP{\c}{\tau}\big(\df\big)^T \Big]^\bot
\\&\quad
  + \df(\tau)\,\k
  - \fg\IP{\cz}{\tau}\k\,.
\end{align*}
\end{lem}


\begin{proof}
Let $\gamma:\S^1\rightarrow\R^n$ be a smooth, closed immersed curve parametrised by arc-length.
Consider the variation $\eta$ of $\gamma$ given by
\[
\eta(s,t) = \gamma(s) + t\phi(s) = \gamma(s) + t\big(V(s)+\varphi(s)\tau(s)\big),
\]
where $\phi(s) = V(s)+\varphi(s)\tau(s)$ with $V$ a vector field normal along $\gamma$ and
$\varphi:\S^1\rightarrow\R$ a function.
Note that $\partial_t\eta = V + \varphi\tau$.
The first variation of an integral $\int_\eta h(s,t)\, ds$, where $h:\S^1\times[0,T)\rightarrow\R$ is some
differentiable function, is given by the formula
\begin{equation}
\label{EQfvintevo}
\rd{}{t}\int_\eta h\,ds\bigg|_{t=0}
 = \int_\gamma (\partial_th)\,ds + \int_\gamma h(\partial_s\varphi - \IP{V}{\k})\,ds\,.
\end{equation}

The evolution of $\tau$ and $\k$ is computed in \cite[Lemma 2.1]{DKS02}:
\begin{align}
\label{EQfvpartialttau}
\partial_t\tau &= \nabla_sV + \varphi\k
\\
\label{EQfvpartialtk}
\partial_t\k &= \nabla_s^2V + \IP{\k}{V}\k + \varphi(\nabla_s\k - |\k|^2\tau) - \IP{\k}{\nabla_sV}\tau\,.
\end{align}
Clearly \eqref{EQfvpartialttau} implies
\begin{equation}
\label{EQfvczevo}
\partial_t\c
 = \dca(V+\varphi\tau) + \df(V+\varphi\tau)\tau + \fg\nabla_sV + \varphi\fg\,\k\,.
\end{equation}

Expanding the square in $\SH$ we find
\begin{equation}
\label{EQfvshdecomp}
\SH(\eta)
  =
     \bigg[ \frac12\int_\eta |\k|^2ds
    + \lambda L(\eta)\bigg]
    + \bigg[\frac12\int_\eta |\c|^2ds\bigg]
    + \bigg[- \int_\eta \IP{\k}{\c}\,ds\bigg]\,.
\end{equation}
We shall compute the first variation of each term in turn.
The first variation of the first term in \eqref{EQfvshdecomp} is computed in \cite{DKS02}; for completeness,
we briefly summarise the computation below.
Using \eqref{EQfvintevo} and then \eqref{EQfvpartialtk} we have
{%
\setlength{\belowdisplayskip}{9pt}%
\setlength{\abovedisplayskip}{9pt}%
\begin{align*}
\rd{}{t}&\Big[\frac12\int_\eta |\k|^2ds + \lambda L(\eta)\Big]\bigg|_{t=0}
  =   \int_\gamma \Big(\frac12|\k|^2 + \lambda\Big)\Big(\partial_s\varphi - \IP{V}{\k}\Big)\,ds
    + \int_\gamma \IP{\k}{\partial_t\k}\,ds
\\
 &=   \int_\gamma \Big(\frac12|\k|^2 + \lambda\Big)\Big(\partial_s\varphi - \IP{V}{\k}\Big)\,ds
    + \int_\gamma \IP{\k}{\nabla_s^2V + \IP{\k}{V}\k + \varphi(\nabla_s\k - |\k|^2\tau)}\,ds
\\
 &= - \int_\gamma \IP{V}{ -\nabla^2_s\k -\frac12|\k|^2\k + \lambda\k}\,ds\,,
\intertext{
where for the last equality we applied integration by parts and the identity
\begin{equation}
\label{EQfvpartialsknablask}
\partial_s\k = \nabla_s\k - |\k|^2\tau\,.
\end{equation}
For the second term of \eqref{EQfvshdecomp}, we use \eqref{EQfvintevo} and \eqref{EQfvczevo} to
compute
}
\rd{}{t}&\Big[\frac12\int_\eta |\c|^2ds \Big]
 = - \int_\gamma \IP{V}{ \frac12|\c|^2\k }\,ds
   + \int_\gamma \IP{\c}{\dca(V) + \df(V)\,\tau + \fg\,\nabla_sV}\,ds
\\
&\quad
   + \frac12\int_\gamma |\c|^2(\partial_s\varphi)\,ds
   + \int_\gamma \IP{\c}{\dca(\varphi\tau) + \df(\varphi\tau)\,\tau + \fg\,(\varphi\k)}\,ds
\\
&= - \int_\gamma \IP{V}{
            \frac12|\c|^2\k
          - (\dca)^T\c - \IP{\c}{\tau}(\df)^T 
                       }\,ds
\\
&\quad
   + \int_\gamma \IP{\c}{
            \fg \big(\partial_sV + \IP{\k}{V}\tau\big)}\,ds\,,
\intertext{
where for the last equality we used integration by parts and the identities
\begin{align}
\nabla_sV
 &= \partial_sV + \IP{\k}{V}\tau
\notag
\\
\partial_s\c
 &= \dca(\tau) + \df(\tau)\,\tau + \fg\,\k\,.
\label{EQpartialscz}
\end{align}
Integrating by parts again on the second term, applying \eqref{EQpartialscz} and rearranging we find
}
\rd{}{t}&\Big[\frac12\int_\eta |\c|^2ds \Big]
= - \int_\gamma \bigg\langle V, 
            \frac12|\c|^2\k
          - (\dca)^T\c - \IP{\c}{\tau}(\df)^T
          + \df(\tau)\,\c
\\*
&\hskip+6cm
          + \fg\big(\dca(\tau) + \df(\tau)\,\tau + \fg\k\big)
          - \fg\IP{\c}{\tau}\k
                       \bigg\rangle\,ds
\\
&= - \int_\gamma \bigg\langle V, 
            \frac12|\c|^2\k
          - (\dca)^T\c - \IP{\c}{\tau}(\df)^T
          + \df(\tau)\,\cz
\\*
&\hskip+6cm
          + 2\df(\tau)\,\tau
          + \fg\dca(\tau)
          - \fg\IP{\cz}{\tau}\k
                       \bigg\rangle\,ds\,.
\intertext{
Continuing with the third term in \eqref{EQfvshdecomp}, we use \eqref{EQfvintevo},
\eqref{EQfvpartialtk}, \eqref{EQfvczevo}, then integration by parts and \eqref{EQfvpartialsknablask},
\eqref{EQpartialscz} to obtain
}
\rd{}{t}&\Big[-\int_\eta \IP{\k}{\c}\,ds \Big]
 =
     \int_\gamma \IP{\k}{\c}\big(\IP{V}{\k}-\partial_s\varphi\big)\,ds
   - \int_\gamma \IP{\k}{\partial_t\c}\,ds
   - \int_\gamma \IP{\c}{\partial_t\k}\,ds
\\
&=
     \int_\gamma \IP{\k}{\c}\IP{V}{\k}\,ds
   - \int_\gamma \IP{\k}{\dca(V) + \df(V)\,\tau + \fg\nabla_sV}\,ds
\\*
&\qquad
   - \int_\gamma \IP{\c}{\nabla_s^2V + \IP{\k}{V}\k - \IP{\k}{\nabla_sV}\tau}\,ds
   - \int_\gamma \IP{\k}{\c}\partial_s\varphi\,ds
\\*
&\qquad
   - \int_\gamma \IP{\k}{\dca(\varphi\tau) + \df(\varphi\tau)\,\tau + \fg\,(\varphi\k)}\,ds
   - \int_\gamma \IP{\c}{\varphi(\nabla_s\k - |\k|^2\tau) } \,ds
\\
&=
   - \int_\gamma \IP{\k}{\dca(V) + \fg\nabla_sV}\,ds
   - \int_\gamma \IP{\c}{\nabla_s^2V - \IP{\k}{\nabla_sV}\tau}\,ds
\\*
&\qquad
   + \int_\gamma \varphi\IP{\partial_s\k}{\c}\,ds
   - \int_\gamma \varphi\IP{\c}{\nabla_s\k - |\k|^2\tau} \,ds
\\
&=
   - \int_\gamma \bigg\langle V, (\dca)^T\k - \fg\nabla_s\k - \df(\tau)\,\k
                      + \nabla_s^2\c
                      + \IP{\c}{\tau}\nabla_s\k
\\*
&\hskip+4cm
                      + \IP{\c}{\k}\k
                      + \IP{\dca(\tau)}{\tau}\k + \df(\tau)\,\k
                       \bigg\rangle\,ds
\\
&=
   - \int_\gamma \bigg\langle V, (\dca)^T\k 
                      + \nabla_s^2\c
                      + \IP{\cz}{\tau}\nabla_s\k
                      + \IP{\cz}{\k}\k
                      + \IP{\dca(\tau)}{\tau}\k
                       \bigg\rangle\,ds\,.
\intertext{Combining these calculations we have}
\rd{}{t}&\SH(\eta)
  =
     \rd{}{t}\bigg[ \frac12\int_\eta |\k|^2ds
    + \lambda\rd{}{t} L(\eta)\bigg]
    + \rd{}{t}\bigg[\frac12\int_\eta |\c|^2ds\bigg]
    + \rd{}{t}\bigg[- \int_\eta \IP{\k}{\c}\,ds\bigg]\,.
\\
&=
- \int_\gamma \IP{V}{ -\nabla^2_s\k -\frac12|\k|^2\k + \lambda\k}\,ds
\\
&\quad
   - \int_\gamma \bigg\langle V, 
            \frac12|\c|^2\k
          - (\dca)^T\c - \IP{\c}{\tau}(\df)^T
          + \df(\tau)\,\cz
\\
&\hskip+6cm
          + 2\df(\tau)\,\tau
          + \fg\dca(\tau)
          - \fg\IP{\cz}{\tau}\k
                       \bigg\rangle\,ds
\\
&\quad
   - \int_\gamma \bigg\langle V, (\dca)^T\k 
                      + \nabla_s^2\c
                      + \IP{\cz}{\tau}\nabla_s\k
                      + \IP{\cz}{\k}\k
                      + \IP{\dca(\tau)}{\tau}\k
                       \bigg\rangle\,ds
\\
&=
- \int_\gamma \bigg\langle V,
                     - \nabla^2_s\k
                     + \IP{\cz}{\tau}\nabla_s\k
                     + \nabla_s^2\c
                     - \frac12|\c-\k|^2\k
          - \fg\IP{\cz}{\tau}\k
\\*
&\hskip+4cm
                     + \big(\lambda+|\c|^2\big)\k
                     + \IP{\dca(\tau)}{\tau}\k
                     + (\dca)^T(\k-\c)
                    \bigg\rangle\,ds
\\*
&\quad
- \int_\gamma \bigg\langle V, 
          - \IP{\c}{\tau}(\df)^T
          + \df(\tau)\,\cz
          + 2\df(\tau)\,\tau
          + \fg\dca(\tau)
                       \bigg\rangle\,ds\,.
\intertext{Now since}
&\hskip+4cm \nabla_s^2\c = \nabla_s^2\cz + \df(\tau)\,\k + \fg\nabla_s\k\,,
\intertext{this simplifies to}
\rd{}{t}&\SH(\eta)
\\*
&=
- \int_\gamma \bigg\langle V,
                     - \nabla^2_s\k
                     + \IP{\c}{\tau}\nabla_s\k
                     + \nabla_s^2\cz
                     - \frac12|\c-\k|^2\k
          - \fg\IP{\cz}{\tau}\k
\\*
&\hskip+4cm
                     + \big(\lambda+|\c|^2\big)\k
                     + \IP{\dca(\tau)}{\tau}\k
                     + (\dca)^T(\k-\c)
                    \bigg\rangle\,ds
\\*
&\quad
- \int_\gamma \bigg\langle V, 
          - \IP{\c}{\tau}(\df)^T
          + \df(\tau)\,(\cz+\k)
          + 2\df(\tau)\,\tau
          + \fg\dca(\tau)
                       \bigg\rangle\,ds\,.
\end{align*}%
}%
Noting that $V$ is normal on $\gamma$, the $L^2$-gradient of $\SH$ is 
\begin{align*}
\BH(\gamma)
&=
   \nabla^2_s\k
   - \IP{\c}{\tau}\nabla_s\k
   - \nabla_s^2\cz
   + \frac12|\c-\k|^2\k
   + \fg\IP{\cz}{\tau}\k
   - \big(\lambda+|\c|^2\big)\k
\\*
&\quad
   - \bigg[(\dca)^T(\k-\c)
      - \IP{\c}{\tau}(\df)^T
      + \df(\tau)\,\cz
      + \fg\dca(\tau)
     \bigg]^\perp
\\*
&\quad
   - \df(\tau)\,\k
   - \IP{\dca(\tau)}{\tau}\k
\,.
\end{align*}
The Euler-Lagrange equation for $\SH$ is $\BH(\gamma)\equiv0$ and the steepest descent $L^2$-gradient flow
with initial data $\gamma_0$ is the one-parameter family of immersed curves
$\gamma:\S^1\times[0,T)\rightarrow\R^n$ satisfying $\gamma(\cdot,0) = \gamma_0$ and $\partial_t\gamma =
-\BH(\gamma)$.

\end{proof}


\section{Families of curves with uniformly bounded $\SH$}

The results of this section (apart from Lemma \ref{LMaprioricontrolofGammaf}, which requires an
additional continuity  assumption) hold for any one-parameter family of closed curves
$\gamma:\S^1\times I\rightarrow\R^n$, $I$ an interval (not necessarily bounded) with uniformly
bounded $\SH$.
In order to remain notationally consistent with the application of these estimates to the generalised Helfrich
flow, we write this uniform bound as
\begin{equation*}
\SH(\gamma(\cdot,t)) \le \SH(\gamma_0)\,,
\end{equation*}
where $\SH(\gamma_0)$ denotes a constant.  In the case of a generalised Helfrich flow, it will denote the
energy of the initial data.
We do not require that the family be differentiable (in time).
Each $\gamma(\cdot,t)$, $t\in I$, need only enough spatial regularity so that $\k\in L^2(\S^1)$ for each $t\in
I$.
This bound is not assumed a-priori to be uniform.
This regularity assumption is the same as $\gamma(\cdot,t)$ being of class $W^{2,2}$ in the arc-length
parametrisation.
Any family satisfying these conditions is termed \emph{$\SH(\gamma_0)$-bounded}.

We begin by demonstrating that every $\SH(\gamma_0)$-bounded family has $\vn{\k}_{L^2}$ and
$\vn{\c}_{L^2}$ bounded \emph{uniformly}.
Assumptions \ref{Aca} and \ref{Af} are typically assumed throughout.

\begin{lem}
\label{LMaprioriL2controlofK}
Let $\gamma:\S^1\times I\rightarrow\R^n$ be a one-parameter family of $\SH(\gamma_0)$-bounded curves.
Suppose $\ca$ and $f$ fulfil Assumptions \ref{Aca} and \ref{Af} respectively and $\lambda>0$.
Then,
\begin{equation}
\label{EQLMaprioriL2controlofKCNCLSN}
\frac12\int_\gamma |\k|^2 ds
+ \frac12\int_\gamma |\c|^2 ds
 \le \SH(\gamma_0).
\end{equation}
\end{lem}
\begin{proof}
By assumption we have
\begin{equation}
\label{EQLMaprioriL2controlofKPF1}
\frac12\int_\gamma |\k-\c|^2 ds
+ \lambda L(\gamma)
\le \SH(\gamma_0).
\end{equation}
Note that
\begin{equation*}
\frac12\int_\gamma |\k-\c|^2 ds
=
  \frac12\int_\gamma |\k|^2 ds
+ \frac12\int_\gamma |\c|^2 ds
- \int_\gamma \IP{\k}{\c} ds
\end{equation*}
and
\begin{equation}
\label{EQLMaprioriL2controlofKPF2}
- \int_\gamma \IP{\k}{\c} ds
=
- \int_\gamma \IP{\partial^2_s\gamma}{\c} ds
=
    \int_\gamma \IP{\tau}{\dca(\tau)} ds
 +  \int_\gamma \df(\tau)\,ds\,.
\end{equation}
Since $\df(\tau) = (f\circ\gamma)'(s)$ and $\gamma$ is closed, the second integral vanishes.
If $\ca$ satisfies Assumption \ref{Aca} with $\L$ positive semi-definite, then
\begin{align*}
  \frac12\int_\gamma |\k|^2 ds
+ \frac12\int_\gamma |\c|^2 ds
+ \lambda L(\gamma)
&\le \SH(\gamma_0)
- \int_\gamma \IP{\tau}{\dca(\tau)} ds
\\
&\le \SH(\gamma_0).
\end{align*}
If $\ca$ instead satisfies Assumption \ref{Aca} with $|\L| \le \lambda$, then
\begin{align*}
  \frac12\int_\gamma |\k|^2 ds
+ \frac12\int_\gamma |\c|^2 ds
+ \lambda L(\gamma)
&\le \SH(\gamma_0)
- \int_\gamma \IP{\tau}{\dca(\tau)} ds
\\
&\le \SH(\gamma_0)
+ \int_\gamma \Big|\dca\Big|\, ds
\\
&\le \SH(\gamma_0)
+ \lambda L(\gamma)
\end{align*}
and subtracting $\lambda L(\gamma)$ from both sides gives \eqref{EQLMaprioriL2controlofKCNCLSN}.
\end{proof}

We now use the uniform bounds on $\vn{\k}_{L^2}$ and $\SH(\gamma)$ to obtain uniform upper and lower bounds
for $L(\gamma)$.

\begin{lem}
\label{LMaprioricontrolofL}
Let $\gamma:\S^1\times I\rightarrow\R^n$ be a one-parameter family of $\SH(\gamma_0)$-bounded curves.
Suppose $\ca$ and $f$ fulfil Assumptions \ref{Aca} and \ref{Af} respectively and $\lambda>0$.
Then,
\begin{equation}
\label{EQLMaprioricontrolofLCNCLSN}
 \frac{2\pi^2}{\SH(\gamma_0)}
\ \le\ L(\gamma)\ \le\ \frac{1}{\lambda}\SH(\gamma_0)\,.
\end{equation}
\end{lem}
\begin{proof}
As $\gamma$ is closed we have $\int_\gamma \tau^i\, ds = 0$ where $\tau^i = \IP{\tau}{e_i}$ and
$\{e_i\}_{i=1}^n$ is an orthonormal basis of $\R^n$.
The standard Poincar\'e inequality then implies
\begin{equation*}
L(\gamma) = \int_\gamma |\tau|^2 ds \le \frac{L(\gamma)^2}{4\pi^2} \int_\gamma |\k|^2 ds.
\end{equation*}
Combining this with Lemma \ref{LMaprioriL2controlofK} gives
\begin{equation*}
L(\gamma) \ge 4\pi^2 \vn{\k}_{L^2}^{-2}
 \ge \frac{2\pi^2}{\SH(\gamma_0)},
\end{equation*}
which is the first inequality in \eqref{EQLMaprioricontrolofLCNCLSN}.
For the second, simply note that \eqref{EQLMaprioriL2controlofKPF1} implies
\begin{equation*}
\lambda L(\gamma)
\le \SH(\gamma_0) -
\frac12\int_\gamma |\k-\c|^2 ds
\le \SH(\gamma_0).
\end{equation*}
\end{proof}

Geometric flows and functionals are typically invariant under translations, and so one may bound the length of
the position vector by translating the origin at each time to any point on the curve and using the inequality
$|\gamma| \le L(\gamma)/2$.
In fact, in the case where $\SH$ is translation-invariant, one can do much better than this, as the proof of
Theorem \ref{TMmt} shows.

If $\cz$ and $\fg$ are not constants, then the functional $\SH$ (and the flow \eqref{EQflow}) is
\emph{not} invariant under translations, and it becomes a nontrivial matter to bound the length of the
position vector along a $\SH(\gamma_0)$-bounded family.
Indeed, as Example \ref{EXtranslatingcircle} shows, such an estimate does not in general hold.

It is possible however to enforce additional restrictions upon $\fg$ and $\cz$ which allow us to uniformly
bound $|\gamma|$ a-priori.
We begin with the case where $\L$ is invertible and non-vanishing.  Here one may perform a direct argument,
and show that there is an absolute bounded radius $\rho$ such that any $\SH(\gamma_0)$-bounded family of
curves remain contained in the ball $B_\rho(0)$.

\begin{lem}
\label{LMaprioricontrolofGamma}
Let $\gamma:\S^1\times I\rightarrow\R^n$ be a one-parameter family of $\SH(\gamma_0)$-bounded curves.
Suppose $\ca$ and $f$ fulfil Assumptions \ref{Aca} and \ref{Af} respectively and $\lambda>0$.
If $\L$ is invertible and non-vanishing, then
\begin{equation}
\label{EQLMaprioricontrolofGammaCNCLSN}
|\gamma|
 \le n\SH(\gamma_0)\bigg(
   \frac{1}{\lambda} + \frac{1}{\pi} |\L^{-1}| \sqrt{(|\M|^2+2c_0^2)\lambda^{-1} + 4}
                 \,\bigg)\,.
\end{equation}
\end{lem}
\begin{proof}
We first briefly note that the assumptions of this lemma allow us to write $\c(s) = \L\gamma(s) + \M +
\fg(s)\tau(s)$ for an invertible $(n\times n)$ constant matrix $\L$ and a constant vector $\M$,
where $\L\not\equiv0$.
Then
\begin{align*}
2|\c|^2
 &= 2|\L\gamma+\M|^2 + 2\fg^2 + 4\fg\IP{\L\gamma+\M}{\tau}
\\
 &\ge |\L\gamma+\M|^2 - 2\fg^2
\\
 &= |\L\gamma|^2 + 2\IP{\M}{\L\gamma} + |\M|^2 - 2\fg^2
\\
 &\ge \frac12|\L\gamma|^2 - |\M|^2 - 2\fg^2
\\
 &\ge \frac12|\L^{-1}|^{-2}|\gamma|^2 - |\M|^2 - 2\fg^2\,,
\end{align*}
which, combined with Lemma \ref{LMaprioriL2controlofK}, yields
\begin{align*}
\frac12|\L^{-1}|^{-2}
\int_\gamma |\gamma|^2 ds
  - |\M|^2 L(\gamma)
  - 2\int_\gamma |\fg|^2ds
 \le
 2\int_\gamma |\c|^2 ds
 \le 4\SH(\gamma_0).
\end{align*}
Rearranging and using Assumption \ref{Af} gives
\begin{equation}
\label{EQL2boundforposvec}
\int_\gamma |\gamma|^2 ds
\le
2 |\L^{-1}|^2 \big(
                    (|\M|^2+2c_0^2) L(\gamma)
                  + 4 \SH(\gamma_0)
              \big)\,.
\end{equation}
Take $\{e_i\}_{i=1}^n$ to be the standard orthonormal basis of $\R^n$.
Then \eqref{EQL2boundforposvec} implies
\begin{equation}
\label{EQL1boundforposvec}
\bigg(\int_\gamma |\IP{\gamma}{e_i}|\,ds\bigg)^2
 \le
 2L(\gamma) |\L^{-1}|^2 \big(
                    (|\M|^2+2c_0^2) L(\gamma)
                  + 4 \SH(\gamma_0)
                        \big)\,.
\end{equation}
Now
\begin{equation*}
\IP{\gamma}{e_i} - \frac1{L(\gamma)}\int_\gamma \IP{\gamma}{e_i}\,ds
\le
\int_\gamma |\IP{\tau}{e_i}|\,ds
\le L(\gamma),
\end{equation*}
so, using \eqref{EQL1boundforposvec},
\begin{align*}
\IP{\gamma}{e_i}
&\le L(\gamma)
 + \frac1{L(\gamma)}\int_\gamma |\IP{\gamma}{e_i}|\,ds
\\
&\le L(\gamma)
 + \frac{\sqrt2}{\sqrt{L(\gamma)}}
|\L^{-1}| \sqrt{
                    (|\M|^2+2c_0^2) L(\gamma)
                  + 4 \SH(\gamma_0)
 }\,.
\end{align*}
Clearly
\begin{equation*}
|\gamma|
\le
  \bigg|\sum_{i=1}^n \IP{\gamma}{e_i}e_i\bigg|
\le nL(\gamma)
 + \frac{n\sqrt2}{\sqrt{L(\gamma)}}
|\L^{-1}| \sqrt{
                    (|\M|^2+2c_0^2) L(\gamma)
                  + 4 \SH(\gamma_0)
 }\,.
\end{equation*}
which, after estimating $L(\gamma)$ with \eqref{EQLMaprioricontrolofLCNCLSN}, is
\eqref{EQLMaprioricontrolofGammaCNCLSN}.

\end{proof}

We are also able to exhibit a-priori control of the position vector in the case where $f$ satisfies Assumption
\ref{ASfproperlike}.
For this we require the family to be continuous, which is of course the case when $\gamma$ is a
generalised Helfrich flow.

\begin{lem}
\label{LMaprioricontrolofGammaf}
Let $\gamma:\S^1\times I\rightarrow\R^n$ be a one-parameter family of $\SH(\gamma_0)$-bounded curves.
Suppose $\ca$ and $f$ fulfil Assumptions \ref{Aca} and \ref{Af}, \ref{ASfproperlike} respectively and
$\lambda>0$.
Then
\begin{equation}
\label{EQLMaprioricontrolofGammafCNCLSN}
|\gamma|
 \le R + (2\lambda)^{-1}\SH(\gamma_0)\,.
\end{equation}
\end{lem}
\begin{proof}
Clearly we have
\begin{align*}
\int_\gamma |\cz\ts|^2ds
 &=
    \int_\gamma |\c|^2ds
  + \int_\gamma \fg \IP{\cz\ts}{\tau}\,ds
  - \int_\gamma |\fg|^2ds
\\
 &\le 
    \int_\gamma |\c|^2ds
  + \frac12\int_\gamma |\cz\ts|^2ds
  - \frac12\int_\gamma |\fg|^2ds
\intertext{and so}
\int_\gamma |\cz\ts|^2ds
 &\le 
    2\int_\gamma |\c|^2ds
  - \int_\gamma |\fg|^2ds\,.
\end{align*}
This implies
\begin{equation}
\label{EQpfLMaprioricontrolGammafestforinnerprod}
\int_\gamma \fg \IP{\cz\ts}{\tau}\,ds
\ge - \int_\gamma |\fg|\,|\cz\ts|\,ds
\ge - \frac12 \int_\gamma |\fg|^2ds - \frac12\int_\gamma|\cz\ts|^2ds
\ge - \int_\gamma |\c|^2ds\,.
\end{equation}
Using \eqref{EQpfLMaprioricontrolGammafestforinnerprod} we bound the energy from below by
\begin{align*}
\SH(\gamma)
 &=
    \frac12\int_\gamma |\k-\cz\ts-\fg\tau|^2ds
  + \lambda L(\gamma)
\\
 &=
    \frac12\int_\gamma |\k-\cz\ts|^2ds
  + \int_\gamma \fg\IP{\cz\ts}{\tau}\,ds
  + \frac12\int_\gamma |\fg|^2ds
  + \lambda L(\gamma)
\\
 &\ge
  - \int_\gamma |\c|^2ds
  + \frac12\int_\gamma |\fg|^2ds\,.
\end{align*}
Let us assume that \eqref{EQLMaprioricontrolofGammafCNCLSN} does not hold for some $t$.
By continuity of the family $\gamma$, there is a smallest $t^*<t$ such that $|\gamma(s,t^*)| \ge R$
for all $s$ and there is an $s^*$ with the property that $|\gamma(s^*,t^*)| = R$.  Since
$\gamma(\cdot,t^*)$ is a regular curve, Lemma \ref{LMaprioricontrolofL} implies that
$\gamma(\cdot,t^*)$ is contained in the closure of the annulus $B_{R+(2\lambda)^{-1}\SH(\gamma_0)}
\tilde B_R$; that is,
\begin{equation}
\label{EQpfLMaprioricontrolGammafforcontraeq}
R \le |\gamma(\cdot,t^*)|
\le R+2\lambda^{-1}\SH(\gamma_0).
\end{equation}
Assumption \ref{ASfproperlike} thus implies
\begin{equation*}
|\fg(s,t^*)| > \frac{\sqrt{3}\SH(\gamma_0)}{\pi}\quad\text{for all $s$}\,.
\end{equation*}
Applying Lemmas \ref{LMaprioriL2controlofK} and \ref{LMaprioricontrolofL} we obtain
\begin{align*}
\SH(\gamma(\cdot,t^*))
 &>
  - 2\SH(\gamma_0) + \frac12L(\gamma) \frac{3\big(\SH(\gamma_0)\big)^2}{\pi^2}
\\
 &\ge
  - 2\SH(\gamma_0)
  + \Big(\frac{2\pi^2}{\SH(\gamma_0)}\Big)\frac{3\big(\SH(\gamma_0)\big)^2}{2\pi^2}
\\
 &\ge \SH(\gamma_0)\,,
\end{align*}
which is a contradiction (note that the first inequality is strict).
Therefore \eqref{EQpfLMaprioricontrolGammafforcontraeq} does not hold and for any given $t$ there exists a
$p\in\S^1$ such that $|\gamma(t,p)| \le R$.  Since $\gamma$ is closed and by Lemma \ref{LMaprioricontrolofL}
we have $L(\gamma) \le \lambda^{-1}\SH(\gamma_0)$ this implies $|\gamma(t,\cdot)| \le R +
(2\lambda)^{-1}\SH(\gamma_0)$ and we are finished.
\end{proof}


With the help of Lemmas \ref{LMaprioricontrolofGamma} and \ref{LMaprioricontrolofGammaf} above, we are able to
obtain convergence the flow and uniqueness of its limit.

In the most general case, this is not possible.
We are only able to control the length of the position vector uniformly on compact subsets
of $I$, and only in the case where the $\SH(\gamma_0)$-bounded family is a generalised Helfrich flow.
This will allow us to obtain the global existence statement of Theorem \ref{TMmt} in full generality, but its
non-uniformality will become a (necessary, see Example \ref{EXtranslatingcircle}) obstacle when we
investigate the asymptotic properties of generalised Helfrich flows.

The strategy we use to obtain the estimate on compact subsets of $I$ is to use the definition of the flow
\eqref{EQflow} and apply the a-priori estimates (proved in the next section) for all derivatives of curvature
to bound $|\partial_t\gamma|$ directly.
Before this can happen however, we must bound $\c$ in $L^\infty$, and since $\cz(s) = \L\gamma(s) + \M$, we
again encounter the problem of bounding the position vector $\gamma$.
In order to circumvent this possible circularity, we shall directly obtain $L^\infty$ control of
$\c$.  This is provided by the following lemma.

\begin{lem}
\label{LMaprioricontrolofC0}
Let $\gamma:\S^1\times I\rightarrow\R^n$ be a one-parameter family of $\SH(\gamma_0)$-bounded curves.
Suppose $\ca$ and $f$ fulfil Assumptions \ref{Aca} and \ref{Af} respectively and $\lambda>0$.
Then,
\begin{equation}
\label{EQLMaprioricontrolofC0CNCLSN}
|\c|
 \le n\lambda^{-1}\SH(\gamma_0)\sqrt{
   8\pi^{-2}\lambda^2 + 3|\L|^2 + 6\lambda (c_0)^2 + 3(c_1)^2
                 }\,.
\end{equation}
\end{lem}
\begin{proof}
We use a strategy similar to that of the previous lemma.
Let $\{e_i\}_{i=1}^n$ be the standard basis of $\R^n$.
Since
\begin{align*}
\bigg|\IP{\c}{e_i} - \frac1{L(\gamma)}\int_\gamma \IP{\c}{e_i}\,ds\bigg|^2
 &\le \bigg(\int_\gamma \big| \L\tau + \df(\tau)\,\tau + \fg\k \big|\, ds\bigg)^2
\\
 &\le
       3L(\gamma)^2\big( |\L|^2 + (c_1)^2 \big)
     + 3(c_0)^2L(\gamma)\int_\gamma |\k|^2ds\,
\end{align*}
we may use Lemmas \ref{LMaprioriL2controlofK} and \ref{LMaprioricontrolofL} to estimate
\begin{align*}
\big|\IP{\c}{e_i}\big|^2
 &\le \frac2{L(\gamma)}\int_\gamma |\c|^2ds
     + 3L(\gamma)^2\big( |\L|^2 + (c_1)^2 \big)
     + 3(c_0)^2L(\gamma)\int_\gamma |\k|^2ds\,
\\
 &\le
      \frac2{\pi^2}\SH(\gamma_0)^2
     + \frac{3}{\lambda^2}\SH(\gamma_0)^2\big( |\L|^2 + (c_1)^2 \big)
     + \frac{6(c_0)^2}{\lambda}\SH(\gamma_0)^2\,.
\end{align*}
This is, upon rearrangement, the claim of the lemma.
\end{proof}

Clearly $|\cz\ts| \le |\c| + |\fg|$, and so one easily obtains an analogue of
\eqref{EQLMaprioricontrolofC0CNCLSN} for $\cz$.
Alternatively, one may carry out the argument of Lemma \ref{LMaprioricontrolofC0} above for $\cz$ to obtain
the following bound, which is slightly better than \eqref{EQLMaprioricontrolofC0CNCLSN} as it does not depend
on $c_1$.

\begin{lem}
\label{LMaprioricontrolofC}
Let $\gamma:\S^1\times I\rightarrow\R^n$ be a one-parameter family of $\SH(\gamma_0)$-bounded curves.
Suppose $\ca$ and $f$ fulfil Assumptions \ref{Aca} and \ref{Af} respectively and $\lambda>0$.
Then,
\begin{equation*}
\label{EQLMaprioricontrolofCCNCLSN}
|\cz\ts|
 \le 2n\SH(\gamma_0)\sqrt{
   \big[c_0/\SH(\gamma_0)\big]^{2} + \pi^{-2} + |\L|^2(2\lambda)^{-2}
                 }\,.
\end{equation*}
\end{lem}
\begin{proof}
Since
\begin{align*}
\bigg|\IP{\cz}{e_i} - \frac1{L(\gamma)}\int_\gamma \IP{\cz}{e_i}\,ds\bigg|^2
 &\le L(\gamma)^2|\L|^2
\end{align*}
we may use Lemmas \ref{LMaprioriL2controlofK} and \ref{LMaprioricontrolofL} to estimate
\begin{align*}
\big|\IP{\cz}{e_i}\big|^2
 &\le
      \frac4{\pi^2}\SH(\gamma_0)^2
     + 4(c_0)^2
     + \frac{|\L|^2}{\lambda^2}\SH(\gamma_0)^2\,.
\end{align*}
This yields the claim of the lemma.
\end{proof}


\section{A-priori estimates for the generalised Helfrich flow}

Theorem \ref{STE} justifies the use of smooth calculations in the derivation of our estimates.
When we use the expression ``Let $\gamma:\S^1\times[0,T)\rightarrow\R^n$ be a generalised Helfrich flow'' we are
invoking Theorem \ref{STE}.
Since along any generalised Helfrich flow we have $\gamma(\cdot,t)\in C^\infty$ and
$$\rd{}{t}\SH(\gamma) = - \int_\gamma |\BH(\gamma)|^2ds \le 0,$$ the one-parameter family $\gamma$
is continuous and $\SH(\gamma_0)$-bounded, and all the results of Section 3 apply.
Our main goal now is to use these to prove a-priori estimates for all derivatives of curvature.
Let us define
\begin{equation*}
V_k := 
 -\nabla^2_s\k -\frac12|\k|^2\k + \lambda\k
\end{equation*}
and
\begin{align*}
V_c &:= 
     \IP{\c}{\tau}\nabla_s\k
   + \nabla_s^2\cz
   + \frac12|\c|^2\k
   + \IP{\cz}{\k}\k
   - \fg\IP{\cz}{\tau}\k
\\*
&\quad
   + \bigg[(\dca)^T(\k-\c)
      - \IP{\c}{\tau}(\df)^T
      + \df(\tau)\,\cz
      + \fg\dca(\tau)
     \bigg]^\perp
\\*
&\quad
   + \df(\tau)\,\k
   + \IP{\dca(\tau)}{\tau}\k
\end{align*}
so that
\begin{equation*}
V := \partial_t\gamma = V_k+V_c\,.
\end{equation*}
It follows from \eqref{EQfvpartialtk} that
\begin{align}
\label{EQnablatk}
\nabla_t\k
 &= \nabla^2_sV + \IP{\k}{V}\k
\\
 &=
  \nabla^2_sV_k + \IP{\k}{V_k}\k
+ \nabla^2_sV_c + \IP{\k}{V_c}\k
\notag\\
 &=
  \Big[- \nabla^4_s\k
  + \lambda\nabla^2_s\k
  + (P_3^2+P_5^0+\lambda P_3^0)(\k)
  \Big]
  +
  \Big[
  \nabla^2_sV_c + \IP{\k}{V_c}\k
  \Big].
\notag
\end{align}
Equation (2.8) from \cite{DKS02} in our setting reads ($\phi$ a normal vector field)
\begin{equation}
\label{EQinterchange}
\nabla_t\nabla_s\phi
 = 
    \nabla_s\nabla_t\phi + \IP{\k}{V}\nabla_s\phi + \IP{\k}{\phi}\nabla_sV - \IP{\nabla_sV}{\phi}\k.
\end{equation}
Taking $\phi = \nabla_s^m\k$ in \eqref{EQinterchange} we obtain
\begin{equation}
\label{EQinterchangegradmk}
\nabla_t\nabla_s^{m+1}\k
 = 
    \nabla_s\nabla_t\nabla_s^m\k
 + \IP{\k}{V}\nabla_s^{m+1}\k + \IP{\k}{\nabla_s^m\k}\nabla_sV - \IP{\nabla_sV}{\nabla_s^m\k}\k.
\end{equation}
One may use \eqref{EQinterchangegradmk} with an induction argument to prove
\begin{equation}
\nabla_t\nabla_s^m\k
 = 
    -\nabla^{m+4}_s\k
    +\lambda\nabla_s^{m+2}\k
  + (P_3^{m+2}+\lambda P_3^m+P_5^m)(\k)
  + \nabla_s^{m+2}V_c
  + P_{2,1}^m(\k;V_c)\,.
\label{EQgeneralevonablask}
\end{equation}



The following lemma provides uniform control of $\vn{\nabla_s\k}_{L^2}$.

\begin{lem}
\label{LMaprioriL2controlofNablaK}
Let $\gamma:\S^1\times[0,T)\rightarrow\R^n$ be a generalised Helfrich flow.
Suppose $\ca$ and $f$ fulfil Assumptions \ref{Aca} and \ref{Af} respectively and $\lambda>0$.
Then there exists an absolute constant $k_1$ such that
\begin{equation}
\label{EQLMaprioriL2controlofNablaKCNCLSN}
\int_\gamma |\nabla_s\k|^2ds \le \int_\gamma |\nabla_s\k|^2ds\bigg|_{t=0} + k_1.
\end{equation}
The constant $k_1$ depends only on $\SH(\gamma_0)$, $\lambda$, $n$, $|\L|$, $c_0$, $c_1$ and $c_2$.
\end{lem}
\begin{proof}
Using the interchange formula \eqref{EQinterchange} and the normal evolution of the curvature
\eqref{EQnablatk} we compute
\begin{align*}
\nabla_t\nabla_s\k
 &= 
    \nabla_s\nabla_t\k
 + \IP{\k}{V}\nabla_s\k + \IP{\k}{\k}\nabla_sV - \IP{\nabla_sV}{\k}\k
\\
 &= 
    \nabla_s\big(\nabla_s^2V + \IP{\k}{V}\k\big)
 + \IP{\k}{V}\nabla_s\k + |\k|^2\nabla_sV - \IP{\nabla_sV}{\k}\k
\\
 &= 
    -\nabla_s^5\k + \lambda\nabla_s^3\k + (P_3^3+\lambda P_3^1+P_5^1)(\k)
\\&\quad
 + \nabla_s^3V_c + \nabla_s\big(\IP{\k}{V_c}\k\big)
 + \IP{\k}{V_c}\nabla_s\k + |\k|^2\nabla_sV_c - \IP{\nabla_sV_c}{\k}\k
\\
 &= 
    -\nabla_s^5\k + \lambda\nabla_s^3\k + (P_3^3+\lambda P_3^1+P_5^1)(\k)
\\&\quad
 + \nabla_s^3V_c
 + \nabla_s\big(|\k|^2V_c\big)
 - 2\IP{\nabla_s\k}{\k}V_c
 + \IP{\nabla_s\k}{V_c}\k
 + 2\IP{\k}{V_c}\nabla_s\k\,.
\end{align*}
Clearly $\IP{\nabla_s\k}{\partial_t\nabla_s\k} = \IP{\nabla_s\k}{\nabla_t\nabla_s\k}$.
Using this and \eqref{EQfvintevo} we compute
\begin{align}
\rd{}{t}&\int_\gamma |\nabla_s\k|^2ds
+ 2\int_\gamma |\nabla_s^{3}\k|^2ds
+ 2\lambda\int_\gamma |\nabla_s^2\k|^2ds
\notag\\
 &=
   2\int_\gamma \IP{\nabla_s\k}{
    (P_3^3+\lambda P_3^1+P_5^1)(\k)
  }
  \,ds
\notag\\
 &\quad
 + 2\int_\gamma \IP{\nabla_s\k}{
   \nabla_s^3V_c
 + \nabla_s\big(|\k|^2V_c\big)
  }
  \,ds
\notag\\
 &\quad
 + 2\int_\gamma \IP{\nabla_s\k}{
 - 2\IP{\nabla_s\k}{\k}V_c
 + \IP{\nabla_s\k}{V_c}\k
 + \IP{\k}{V_c}\nabla_s\k\,
  }
  \,ds
\notag\\
 &=
   2\int_\gamma (P_4^{4,2}+\lambda P_4^{2,1}+P_6^{2,1})(\k)
  \,ds
 - 2\int_\gamma \IP{\nabla_s^2\k}{
   \nabla_s^2V_c
 + |\k|^2V_c
  }
  \,ds
\notag\\
 &\quad
 + 2\int_\gamma \IP{\nabla_s\k}{
 - 2\IP{\nabla_s\k}{\k}V_c
 + \IP{\nabla_s\k}{V_c}\k
 + \IP{\k}{V_c}\nabla_s\k\,
  }
  \,ds\,.
\label{EQevogradsinL2}
\end{align}
For the last equality we used integration by parts on the first term to limit the maximum order of
differentiation as follows:
\begin{align*}
\int_\gamma \nabla_s\k*P_3^{3,3}(\k)\, ds
   &=  c\int_\gamma \nabla_s\k*\nabla_s^3\k*\k*\k\, ds
   + \int_\gamma P_4^{4,2}(\k)\, ds
\\ &= -c\int_\gamma \nabla^2_s\k*\big(\nabla_s^2\k*\k*\k+2\nabla_s\k*\nabla_s\k*\k\big)\, ds
   + \int_\gamma P_4^{4,2}(\k)\, ds
\\ &=
     \int_\gamma P_4^{4,2}(\k)\, ds\,.
\end{align*}
We recall the interpolation inequality \cite[(2.16)]{DKS02}.
Let $\nu,\mu, k$ be positive integers with $k\le\mu$.
Suppose $\sigma := \frac1k(\mu+\frac12\nu-1) < 2$.
Then for any $\varepsilon > 0$ there exists an absolute constant $c$ such that the inequality
\begin{equation}
\label{EQDKSinterpfork}
\int_\gamma |P_{\nu}^{\mu,k-1}(\k)|\,ds
 \le
   \varepsilon\int_\gamma |\nabla_s^k\k|^2ds
 + c\varepsilon^{-\frac{\sigma}{2-\sigma}}\bigg(\int_\gamma |\k|^2ds\bigg)^{\frac{\nu-\sigma}{2-\sigma}}
 + c\bigg(\int_\gamma |\k|^2ds\bigg)^{\mu+\nu-1}
\end{equation}
holds.
Applying \eqref{EQDKSinterpfork} three times with
$[\nu=4,\mu=4,k=3,\sigma=\frac{5}{3}<2]$,
$[\nu=4,\mu=2,k=3,\sigma=1<2]$, and
$[\nu=6,\mu=2,k=3,\sigma=\frac{4}{3}<2]$,
we obtain the estimate (cf. \cite[Theorem 3.2]{DKS02})
\begin{align*}
2&\int_\gamma (P_4^{4,2}+\lambda P_4^{2,1}+P_6^{2,1})(\k)\,ds
 = 2\int_\gamma (P_4^{4,2}+\lambda P_4^{2,2}+P_6^{2,2})(\k)\,ds
\\
 &\le 6\varepsilon\int_\gamma|\nabla_s^3\k|^2ds
 + c(\varepsilon^{-5}+\varepsilon^{-2}+1)\bigg(\int_\gamma |\k|^2ds\bigg)^{7}
 + c\varepsilon^{-1}\bigg(\int_\gamma |\k|^2ds\bigg)^{3}
 + c\bigg(\int_\gamma |\k|^2ds\bigg)^{5}
\end{align*}
where $c$ depends only on $\lambda$.
Choosing $\varepsilon = \frac{1}{24}$ gives
\begin{equation}
\label{EQinterpwillmoreterms1}
2\int_\gamma (P_4^{4,2}+\lambda P_4^{2,1}+P_6^{2,1})(\k)\,ds
 \le \frac14\int_\gamma|\nabla_s^3\k|^2ds
 + c\sum_{i=1}^3\bigg(\int_\gamma |\k|^2ds\bigg)^{2i+1}.
\end{equation}
Combining \eqref{EQinterpwillmoreterms1} with \eqref{EQevogradsinL2}, absorbing $\vn{\nabla_s^3\k}_{L^2}^2$ on
the left and estimating $\vn{\k}_{L^2}^2$ by Lemma \ref{LMaprioriL2controlofK} we have
\begin{align}
\rd{}{t}&\int_\gamma |\nabla_s\k|^2ds
+ \frac{7}{4}\int_\gamma |\nabla_s^{3}\k|^2ds
+ 2\lambda\int_\gamma |\nabla_s^2\k|^2ds
\notag\\
 &\le
 - 2\int_\gamma \IP{\nabla_s^2\k}{
   \nabla_s^2V_c
  }
  \,ds
 + c\int_\gamma \big(
           P_3^{2,2}(\k)*V_c
  \big)\,ds
 + c\,,
\label{EQevogradsinL2afteroneabsorbing}
\end{align}
where $c$ depends only on $\lambda$ and $\SH(\gamma_0)$.
As the first step in the proof of \eqref{EQDKSinterpfork} is to apply a H\"older inequality, we observe that
the following inequality which is slightly stronger than \eqref{EQDKSinterpfork} holds:
\begin{align}
\int_\gamma
   |\nabla_s^{\mu_1}\k|\,
  &|\nabla_s^{\mu_2}\k|\,
   \cdots 
   |\nabla_s^{\mu_\nu}\k|\,ds
\notag\\&
 \le
   \varepsilon\int_\gamma |\nabla_s^k\k|^2ds
 + c\varepsilon^{-\frac{\sigma}{2-\sigma}}\bigg(\int_\gamma |\k|^2ds\bigg)^{\frac{\nu-\sigma}{2-\sigma}}
 + c\bigg(\int_\gamma |\k|^2ds\bigg)^{\mu+\nu-1},
\label{EQstrongerDKSinterpfork}
\end{align}
where $\sum_{i=1}^\nu \mu_i = \mu$, $\mu_i \le k-1$ and other notation is as in \eqref{EQDKSinterpfork}.
In what follows we shall use the notation $\P_\nu^{\mu,k}(\k)$ to denote any linear combination of terms of
the type of the integrand on the left hand side of \eqref{EQstrongerDKSinterpfork} with universal, constant
coefficients.

Since
\begin{align*}
\nabla_s^2\cz
 = \nabla_s\big(\L\tau - \IP{\tau}{\L\tau}\tau\big)
 = (\L\k)^\bot - \IP{\tau}{\L\tau}\k,
\end{align*}
one finds that
\begin{align*}
|V_c| &= 
    \bigg|\IP{\c}{\tau}\nabla_s\k
  + (\L\k)^\bot - \IP{\tau}{\L\tau}\k
  + \frac12|\c|^2\k
  + \IP{\c}{\k}\k
  + \IP{\L\tau}{\tau}\k
  + \df(\tau)\,\k
\\ &\quad
  - \fg\IP{\cz}{\tau}\k
   + \Big[\L^T(\k-\c) \Big]^\bot
   + \df(\tau)\big[\cz\ts\big]^\bot
   + \fg\big[\L\tau\big]^\bot
   - \IP{\c}{\tau}\Big[\big(\df\big)^T\Big]^\bot\bigg|
\\
&\le
       |\c| |\nabla_s\k|
 + \Big(
    6|\L| + \frac12|\c|^2 + |\c||\k| + \big|\df\big| + |\fg||\cz\ts|
   \Big)|\k|
\\
&\quad
 + 2\big(|\c| + |\fg|\big)|\L| + 2\big(|\c| + |\cz\ts|\big)\big|\df\big|\,.
\end{align*}
Therefore
\begin{equation}
|V_c|
 \le c\big(
          |\c| + |\c|^2 + |\L| + c_0|\cz\ts| + c_1
      \big)
      \big(
          |\nabla_s\k| + |\k| + |\k|^2 + |\L| + c_0 + 1
      \big)\,.
\label{EQestzeroorderforVc}
\end{equation}
Using Lemma \ref{LMaprioricontrolofC0}, Lemma \ref{LMaprioricontrolofC}, and estimating $2|\k| \le 1 +
|\k|^2$, we obtain
\begin{equation}
\label{EQzeroorderestforVc}
|V_c|
 \le c\big(1+|\k|^2+|\nabla_s\k|\big)\,,
\end{equation}
where $c$ depends only on $\SH(\gamma_0)$, $\lambda$, $n$, $|\L|$, $c_0$, and $c_1$.
From \eqref{EQzeroorderestforVc} we have
\begin{equation*}
 c\int_\gamma \big(
           P_3^{2,2}(\k)*V_c
  \big)\,ds
\le c\int_\gamma \big(\P_3^{2,2} + \P_5^{2,2} + \P_4^{3,2}\big)(\k)\,ds\,.
\end{equation*}
Applying \eqref{EQstrongerDKSinterpfork} three times with
$[\nu=3,\mu=2,k=3,\sigma=\frac{4+1}{6}<2]$,
$[\nu=5,\mu=2,k=3,\sigma=\frac{4+3}{6}<2]$, 
$[\nu=4,\mu=3,k=3,\sigma=\frac{3+1}{3}<2]$,
and using Lemma \ref{LMaprioriL2controlofK} to control $\vn{\k}_{L^2}$, we estimate
\begin{equation*}
 c\int_\gamma \big(
           P_3^{2,2}(\k)*V_c
  \big)\,ds
\le
   \frac14\int_\gamma |\nabla_s^3\k|^2ds
 + c\,,
\end{equation*}
which upon combination with \eqref{EQevogradsinL2afteroneabsorbing}
 implies
\begin{align}
\rd{}{t}\int_\gamma |\nabla_s\k|^2ds
+ \frac{3}{2}\int_\gamma |\nabla_s^{3}\k|^2ds
+ 2\lambda\int_\gamma |\nabla_s^2\k|^2ds
 &\le
 - 2\int_\gamma \IP{\nabla_s^2\k}{
   \nabla_s^2V_c }
  \,ds
 + c\,
\notag\\
 &=
   2\int_\gamma \IP{\nabla_s^3\k}{
   \nabla_sV_c }
  \,ds
 + c\,.
\label{EQevogradsinL2aftertwiceabsorbing}
\end{align}
We must finally estimate the term $\int_\gamma \IP{\nabla_s^3\k}{ \nabla_sV_c } \,ds$.
We begin by computing $\nabla_sV_c$:
\begin{align*}
\nabla_sV_c
 &=
 \Big(\IP{\c}{\tau} \nabla^2_s\k
 + \big(\IP{\L\tau}{\tau} + \df(\tau) + \IP{\cz}{\k}\big)\nabla_s\k\Big)
 + \nabla^3_s\cz
\\
&\quad
 + \Big(
     \frac12|\c|^2\nabla_s\k
   + \IP{\c}{\L\tau + \df(\tau)\,\tau + \fg\k}\k
   \Big)
\\
&\quad
 + \Big(
     \IP{\cz}{\k}\nabla_s\k
   + \IP{\L\tau}{\k}\k
   + \IP{\cz}{\nabla_s\k}\k
   - |\k|^2\IP{\cz}{\tau}\k
   \Big)
\\
&\quad
 + \Big(
      \IP{\L\tau}{\tau}\nabla_s\k
    + \IP{\L\k}{\tau}\k
    + \IP{\L\tau}{\k}\k
   \Big)
\\
&\quad
 + \Big(
     \df(\tau)\,\nabla_s\k
   + \dkf{2}(\tau,\tau)\,\k
   + \df(\k)\,\k
   \Big)
\\
&\quad
 - \Big(
   \fg\IP{\cz}{\tau}\nabla_s\k
 + \df(\tau)\IP{\cz}{\tau}\k
 + \fg\IP{\L\tau}{\tau}\k
 + \fg\IP{\cz}{\k}\k
   \Big)
\\
&\quad
   + \bigg[
      \Big(
        \L^T\big(\nabla_s\k - |\k|^2\tau - \L\tau - \df(\tau)\,\tau - \fg\k\big)
      - \IP{\L^T(\partial_s\k-\partial_s\c)}{\tau}\tau
      \Big)
\\
&\qquad
      - \Big(
          \IP{\L\tau}{\tau}(\df)^T
        + \df(\tau)(\df)^T
        + \IP{\c}{\k}(\df)^T
\\
&\qquad\qquad
        + \IP{\c}{\tau}\big(\dkf{2}(\tau)\big)^T
        - \IP{\c}{\tau}\IP{\big(\dkf{2}(\tau)\big)^T}{\tau}\tau
        \Big)
\\
&\qquad
      + \Big(
            \dkf{2}(\tau,\tau)\,\cz
          + \df(\k)\,\cz
          + \df(\tau)\L\tau
          - \df(\tau)\IP{\L\tau}{\tau}\tau
        \Big)
\\
&\qquad
      + \Big(
            \df(\tau)\L\tau
          + \fg\L\k
          - \fg\IP{\L\k}{\tau}\tau
        \Big)
\\
&\qquad
   - \IP{\L^T(\k-\c)
      - \IP{\c}{\tau}(\df)^T
      + \df(\tau)\,\cz
      + \fg\L\tau}{\tau}\k
     \bigg]
\\&=
      \IP{\c}{\tau} \nabla_s^2\k
    + \nabla^3_s\cz
    + \Big[
           2\IP{\L\tau}{\tau}
         + 2\df(\tau)
         + 2\IP{\cz}{\k}
         + \frac12|\cz\ts|^2
         + \frac12|\fg|^2
      \Big] \nabla_s\k
\\&
\quad
    + \Big[
           \IP{\cz}{\nabla_s\k}
         - |\k|^2\IP{\cz}{\tau}
         + \fg\IP{\cz}{\k}
         + 2\IP{\L\tau}{\k}
         + \IP{\L\k}{\tau}
	 - \df(\tau)\IP{\cz}{\tau}
\\
&\qquad
         + \df(\k)
         - \fg\IP{\L\tau}{\tau}
         - \fg\IP{\cz}{\k}
         + \IP{\c}{\L\tau}
         + \df(\tau)\IP{\c}{\tau}
         + \dkf{2}(\tau,\tau)
\\
&\qquad
   + \IP{ - \L^T\k + \L^T\c
      + \IP{\c}{\tau}(\df)^T
      - \df(\tau)\,\cz
      - \fg\L\tau}{\tau}
      \Big]\k
\\
&\quad
   + \bigg[
          \L^T\nabla_s\k
        - |\k|^2\L^T\tau
        - \fg\L^T\k
        - \IP{\cz}{\k}(\df)^T
        + \df(\k)\,\cz
        + \fg\L\k
        - \L^T\L\tau
\\
&\qquad
        - \df(\tau)\L^T\tau
        - \IP{\L\tau}{\tau}(\df)^T
        - \df(\tau)(\df)^T
        - \IP{\c}{\tau}\big(\dkf{2}(\tau)\big)^T
\\
&\qquad
        + \dkf{2}(\tau,\tau)\,\cz
        + 2\df(\tau)\L\tau
     \bigg]
 + Y\tau
\end{align*}
where 
\begin{align*}
Y &=
      - \IP{\L^T(\partial_s\k-\partial_s\c)}{\tau}
      + \IP{\c}{\tau}\IP{\big(\dkf{2}(\tau)\big)^T}{\tau}
      - \fg\IP{\L\k}{\tau}
      - \df(\tau)\IP{\L\tau}{\tau}
\,.
\end{align*}
In order to control this rather daunting expression let us introduce another kind of $P$-style notation.
We use $P(v_1;\cdots;v_m)$ to denote a polynomial in $v_i$, $\L v_i$, $\L^Tv_i$, for
$i=1,\ldots,m$, of arbitrarily high (but finite) order and with coefficients depending only on universal
constants.
More precisely,
\begin{align}
P(v_1;\cdots;v_m)
 = \sum_{i=1}^p c_{i}\,
                     &\Big(v_{1}^{\alpha_{i,1}} * (\L v_{1})^{\alpha_{i,2}}
                                                       * (\L^Tv_{1})^{\alpha_{i,3}}\Big)
\notag\\
                     &*\Big(v_{2}^{\alpha_{i,3}} * (\L v_{2})^{\alpha_{i,4}}
                                                       * (\L^Tv_{2})^{\alpha_{i,5}}\Big)
\notag\\
                     &*\cdots 
                     *\Big(v_{m}^{\alpha_{i,3m}} * (\L v_{m})^{\alpha_{i,3m+1}}
                                                        * (\L^Tv_{m})^{\alpha_{i,3m+2}}\Big)
\label{EQDFNpolyterm}
\end{align}
for some positive integer $p$, constants $c_{i}\in\R$, and non-negative integeral powers
$\alpha_{i,j}$.
Recall that the $*$ product allows arbitrary re-orderings of the arguments (see \eqref{EQstarnotation}).
In the above expression the $*$-notation has been extended to allow powers of elements, which are expanded
according to:
\begin{equation*}
v^0 = \text{id},\qquad
v^1 = v,\qquad
v^q = v*v^{q-1}\,,\quad q\ge1.
\end{equation*}
It is important to note that there are no derivatives of any $v_i$ in $P(v_1;\cdots;v_m)$.
We also introduce the $\bs$ product, which is an extension of the $*$ product blind to the presence of
premultiplication by $\L$ and $\L^T$; that is, for vectors $X,Y$ let us set
\begin{equation*}
X \bs Y =
     X * (\sigma_1 Y + \sigma_2 \L Y + \sigma_3 \L^TY)
 +  (\sigma_4 X + \sigma_5 \L X + \sigma_6 \L^TX) * Y\,,
\end{equation*}
where $\sigma_i \in \R$ are (possibly zero) constants.

We briefly compute
\begin{align*}
 \nabla_s^3\cz
 &= \nabla_s\Big((\L\k)^\bot - \IP{\tau}{\L\tau}\k\Big)
  = \nabla_s\Big(\L\k - \IP{\tau}{\L\k}\tau - \IP{\tau}{\L\tau}\k\Big)
\\
 &= \L\partial_s\k - \IP{\tau}{\L\tau}\nabla_s\k
  - \big(2\IP{\tau}{\L\k} + \IP{\k}{\L\tau}\big)\k
  - \IP{\tau}{\L\partial_s\k}\tau
\\
 &= \L\nabla_s\k - \IP{\tau}{\L\tau}\nabla_s\k
  - \big(2\IP{\tau}{\L\k} + \IP{\k}{\L\tau}\big)\k
  - \big(\IP{\tau}{\L\partial_s\k} + |\k|^2\big)\tau
\\
 &= P(\tau)\bs(\nabla_s\k + \k\bs\k)
  - \big(\IP{\tau}{\L\partial_s\k} + |\k|^2\big)\tau\,.
\end{align*}
Components of $\nabla_sV_c$ in purely tangential directions (those contained in $Y$ above) will be ignored, as
they vanish upon taking the inner product with $\nabla_s^3\k$.
We collect the remaining terms roughly according to their order (adding $-(\IP{\tau}{\L\partial_s\k} +
|\k|^2)$ to $Y$) by
\begin{align}
\nabla_sV_c
 &=
 \IP{\c}{\tau} \nabla^2_s\k
 + P(\tau;\cz;\fg;\df)\bs(\nabla_s\k + \k\bs\nabla_s\k)
\notag\\
&\quad
 + P\big(\tau;\cz;\fg;\df;\M;\dkf{2}(\tau)\big)\bs(\k + \k\bs\k + \k\bs\k\bs\k)
\notag\\
&\quad
 + P\Big(\tau;\L\tau;\cz;\fg;\df;\M;\dkf{2}(\tau)\Big)
 + Y\tau\,.
\label{EQnablasVcshort}
\end{align}
(Recall that the $*$ product acts on functions, vector fields, and 1-forms.)
Lemmas \ref{LMaprioriL2controlofK}, \ref{LMaprioricontrolofC0}, and \ref{LMaprioricontrolofC} allow us
pointwise control of $\tau$, $\c$, and $\cz$.
Furthermore, Assumption \ref{Af} gives uniform bounds on $\dkf{k}$ for all $k\ge0$.
We may thus estimate 
\[
\bigg|
 P\Big(\tau;\L\tau;\cz;\fg;\df;\M;\dkf{2}(\tau)\Big)
\bigg| \le c.
\]
In the above estimate (and for the remainder of the proof) $c$ depends additionally on $c_2$.
Inserting the expansions above and estimating, we find
\begin{align*}
 2\int_\gamma \IP{\nabla_s^3\k}{
   \nabla_sV_c }
  \,ds
&\le
   c\int_\gamma
        |\nabla_s^3\k|\,
        \Big(|\nabla^2_s\k| + |\nabla_s\k|
           + |\k|\,|\nabla_s\k| + 1 + |\k| + |\k|^2 + |\k|^3\Big)\,ds
\notag\\
&\le
    \frac14\int_\gamma |\nabla_s^3\k|^2ds
  + c\int_\gamma \Big( |\nabla_s^2\k|^2 + |\k|^2|\nabla_s\k|^2 + |\k|^6 \Big)\,ds
  + cL(\gamma)\,.
\end{align*}
Lemma \ref{LMaprioricontrolofL} provides a uniform estimate for the last term on the right, whereas for the
second term we use \eqref{EQDKSinterpfork} (or \eqref{EQstrongerDKSinterpfork}) with
$[\nu=2,\mu=4,k=3,\sigma=\frac43<2]$,
$[\nu=4,\mu=2,k=3,\sigma=1<2]$, and
$[\nu=6,\mu=0,k=3,\sigma=\frac{2}{3}<2]$,
and Lemma \ref{LMaprioriL2controlofK} to obtain
\begin{align*}
   c\int_\gamma \Big( &|\nabla_s^2\k|^2 + |\k|^2|\nabla_s\k|^2 + |\k|^6 \Big)\,ds
 = c\int_\gamma \Big( \P_2^{4,2} + \P_4^{2,2} + \P_6^{0,2} \Big)(\k)\,ds
\\
 &\le \frac14\int_\gamma |\nabla_s^3\k|^2ds + c\,.
\end{align*}
Inserting the above pair of estimates into \eqref{EQevogradsinL2aftertwiceabsorbing} and absorbing yields
\begin{equation}
\label{EQevogradsinL2afterallabsorbing}
\rd{}{t}\int_\gamma |\nabla_s\k|^2ds
+ \int_\gamma |\nabla_s^{3}\k|^2ds
+ \lambda\int_\gamma |\nabla_s^2\k|^2ds
 \le c\,.
\end{equation}
The elementary interpolation inequality ($m\ge1$, $p\ge0$, $m$,$p$ integers)
\begin{equation}
\int_\gamma |\nabla_s^{m}\k|^2ds
\le
\varepsilon\int_\gamma |\nabla_s^{m+p}\k|^2d
+ c_{\varepsilon}\int_\gamma |\k|^2ds
\label{EQsimpleinterp}
\end{equation}
with $m=1$, $p=2$, $\varepsilon=1$, combined with \eqref{EQevogradsinL2afterallabsorbing}, implies
\begin{equation}
\rd{}{t}\int_\gamma |\nabla_s\k|^2ds
 \le c - \int_\gamma |\nabla_s\k|^2ds\,.
\label{EQfinalevofornablak}
\end{equation}
We are now in a position to conclude \eqref{EQLMaprioriL2controlofNablaKCNCLSN} via a simple proof by
contradiction.
Indeed, assuming a bound of the form \eqref{EQLMaprioriL2controlofNablaKCNCLSN} did not hold, for any
$C<\infty$ there would exist a $\tilde{t}$ depending only on $C$ such that
\begin{equation}
\int_\gamma |\nabla_s\k|^2ds
 > C\text{, for all }t\in[\tilde{t},T)\,.
\label{EQcontradassumption}
\end{equation}
This is in particular true for $C = c$ where $c$ the constant in \eqref{EQfinalevofornablak}.
Since $\vn{\nabla_s\k}_{L^2}^2\in C^1([0,T))$, the estimate \eqref{EQfinalevofornablak} applies, which for
$t\in[\tilde{t},T)$ implies
\begin{equation*}
\rd{}{t}\int_\gamma |\nabla_s\k|^2ds
 \le c - \int_\gamma |\nabla_s\k|^2ds
 < 0,
\end{equation*}
in contradiction with \eqref{EQcontradassumption}.
This argument establishes the bound \eqref{EQLMaprioriL2controlofNablaKCNCLSN} with $k_1 = c -
\vn{\nabla_s\k}_{L^2}^2\big|_{t=0}$. 
\end{proof}

We now bound all higher order derivatives of the curvature.

\begin{lem}
\label{LMhigherorderbounds}
Let $\gamma:\S^1\times[0,T)\rightarrow\R^n$ be a generalised Helfrich flow.
Suppose $\ca$ and $f$ fulfil Assumptions \ref{Aca} and \ref{Af} respectively and $\lambda>0$.
Let $m$ be a positive integer.
Then there exist constants $k_m$ such that
\begin{equation}
\label{EQhigherorderboundsCNCLSN}
\int_\gamma |\nabla_s^m\k|^2ds \le \int_\gamma |\nabla_s^m\k|^2ds\bigg|_{t=0} + k_m.
\end{equation}
The constants $k_m$ depend only on $\SH(\gamma_0)$, $\lambda$, $n$, $|\L|$, and $c_i$ for $i = 0,\ldots,m+1$.
\end{lem}
\begin{proof}
As in the previous lemma, we shall employ the notation $P(v_1;\cdots;v_m)$ to denote a polynomial in
$v_i$, $\L v_i$, $\L^Tv_i$.
This notation is particularly helpful in light of Lemmas
\ref{LMaprioriL2controlofK}--\ref{LMaprioriL2controlofNablaK}.
Recalling \cite[Lemma 2.7]{DKS02}, we note the bound $\vn{\k}_{L^\infty} \le c$ and so
\begin{equation*}
|P(\cz;\fg;\tau;\k)| =
|P(\cz;\fg;\partial_s\gamma;\partial_s^2\gamma)| \le c\,.
\end{equation*}
Using \eqref{EQgeneralevonablask} one may compute
\begin{align}
\rd{}{t}&\int_\gamma |\nabla_s^m\k|^2ds
+ 2\int_\gamma |\nabla_s^{m+2}\k|^2ds
+ 2\lambda\int_\gamma |\nabla_s^{m+1}\k|^2ds
+ 2\lambda\int_\gamma |\nabla_s^{m}\k|^2|\k|^2ds
\notag\\
 &=
   2\int_\gamma \IP{\nabla_s^m\k}{
    (P_3^{m+2}+\lambda P_3^m+P_5^m)(\k)}\,ds
 + 2\int_\gamma \nabla_s^m\k * P_{2,1}^m(\k;V_c)\,ds
\notag\\&\quad
\label{EQhigherorderL2nablamskevo}
 + 2\int_\gamma \IP{\nabla_s^m\k}{\nabla_s^{m+2}V_c}\,ds
 -  \int_\gamma |\nabla_s^m\k|^2\IP{\k}{V_c}\,ds\,.
\end{align}
The first term on the right may be interpolated exactly as in \cite[Theorem 3.2]{DKS02}:
\begin{equation}
2\int_\gamma \IP{\nabla_s^m\k}{
    (P_3^{m+2}+\lambda P_3^m+P_5^m)(\k)}\,ds
\le
    \frac14\int_\gamma |\nabla_s^{m+2}\k|^2ds
    + c\,.
\label{EQhigherorderest1}
\end{equation}
For the fourth term we use Lemma \ref{LMaprioriL2controlofNablaK} in combination with the estimate
\eqref{EQestzeroorderforVc} to obtain $|V_c| \le c(1+|\nabla_s\k|)$, which implies
\begin{equation*}
 -  \int_\gamma |\nabla_s^m\k|^2\IP{\k}{V_c}\,ds
\le
    c\int_\gamma |\nabla_s^m\k|^2\,(1+|\nabla_s\k|)\,ds
\le
    c\int_\gamma \big|(P_3^{2m+1;m}+P_2^{2m;m})(\k)\big|\,ds\,.
\end{equation*}
Employing now the interpolation inequality \eqref{EQstrongerDKSinterpfork} with
[$\nu=3$, $\mu=2m+1$, $k=m+2$, $\sigma=\frac{4m+3}{2m+4}<2$] and
[$\nu=2$, $\mu=2m$, $k=m+2$, $\sigma=\frac{2m}{m+2}<2$],
we find
\begin{equation}
 -  \int_\gamma |\nabla_s^m\k|^2\IP{\k}{V_c}\,ds
\le
    \frac14\int_\gamma |\nabla_s^{m+2}\k|^2ds
    + c\,.
\label{EQhigherorderest2}
\end{equation}
Combining \eqref{EQhigherorderest1} and \eqref{EQhigherorderest2} with \eqref{EQhigherorderL2nablamskevo} we
find
\begin{align}
\rd{}{t}&\int_\gamma |\nabla_s^m\k|^2ds
+ \frac32\int_\gamma |\nabla_s^{m+2}\k|^2ds
\notag\\
 &\le
   2\int_\gamma \nabla_s^m\k * P_{2,1}^m(\k;V_c)\,ds
 + 2\int_\gamma \IP{\nabla_s^m\k}{\nabla_s^{m+2}V_c}\,ds\,.
\label{EQhigherorderevo1}
\end{align}
Let us first consider the case $m=2$.
This will serve to demonstrate aspects of the more general method, which we will be able to apply once we can
safely assume $m\ge3$.
Note that \eqref{EQhigherorderboundsCNCLSN} for $m=2$ will imply $L^\infty$ bounds for $\nabla_s\k$ and
thus allow us to control $V_c$ in $L^\infty$.

We need to estimate $\int_\gamma \IP{\nabla_s^4\k}{\nabla_s^2V_c}\,ds$.
Let us first note the formulae ($\phi$, $\eta$ normal)
\begin{align}
\nabla_s\big(P(X;\tau) \bs \phi\big)
 &= P(\partial_sX;X;\tau;\k) \bs (\nabla_s\phi + \phi)\,,
\label{EQdiffrulebs}
\\
\nabla_s\big(\eta \bs \phi\big)
 &=
     \nabla_s\eta \bs \phi + \nabla_s\phi \bs \eta
   - \IP{\phi}{\k}\tau \bs \eta - \IP{\eta}{\k}\tau \bs \phi 
   + Y\tau\,,
\label{EQdiffrunebsprod}
\end{align}
where $Y = 
   - \IP{\nabla_s\eta \bs \phi}{\tau}
   + \IP{\eta}{\k}\IP{\tau \bs \phi}{\tau}
   - \IP{\nabla_s\phi \bs \eta}{\tau}
   + \IP{\phi}{\k}\IP{\tau \bs \eta}{\tau}$.
The second formula \eqref{EQdiffrunebsprod} follows from the computation
\begin{align*}
\nabla_s\big(\eta \bs \phi\big)
 &= \partial_s\eta \bs \phi + \partial_s\phi \bs \eta
   - \IP{\partial_s\eta \bs \phi + \partial_s\phi \bs \eta}{\tau}\tau
\\
 &=  \nabla_s\eta \bs \phi - \IP{\eta}{\k}\tau \bs \phi 
   + \nabla_s\phi \bs \eta - \IP{\phi}{\k}\tau \bs \eta 
   - \IP{\partial_s\eta \bs \phi + \partial_s\phi \bs \eta}{\tau}\tau
\\
 &=  \nabla_s\eta \bs \phi + \nabla_s\phi \bs \eta
   - \IP{\phi}{\k}\tau \bs \eta - \IP{\eta}{\k}\tau \bs \phi 
\\
&\quad
   - \IP{\nabla_s\eta \bs \phi}{\tau}\tau
   + \IP{\eta}{\k}\IP{\tau \bs \phi}{\tau}\tau
   - \IP{\nabla_s\phi \bs \eta}{\tau}\tau
   + \IP{\phi}{\k}\IP{\tau \bs \eta}{\tau}\tau
\end{align*}
(Keep in mind that the $*$ product (and so also the $\bs$ product) is blind to reordering of its arguments.)
In our current situation with $\k\in L^\infty$, this can be usefully cast as
\begin{equation*}
\nabla_s\big(\eta \bs \phi\big)
 = P(\tau;\k)\bs(\phi\bs\eta + \nabla_s\eta \bs \phi + \nabla_s\phi \bs \eta)
\,.
\end{equation*}
This is the form in which we shall apply \eqref{EQdiffrunebsprod} below.
For the proof of \eqref{EQdiffrulebs}, note that $\partial_s P(X) = P(\partial_sX;X)$, $\nabla_s P(X) =
P(\partial_sX;X;\tau)$, $\partial_s\phi = \nabla_s\phi - \phi*\k*\tau$, and compute
\begin{align*}
\nabla_s\big(P(X;\tau) \bs \phi\big)
 &=  P(\partial_sX;X;\tau;\k) \bs \phi
   + P(X;\tau;\k) \bs (\partial_s\phi)
\\
&\quad
   - \Big(P(\partial_sX;X;\tau;\k) \bs \phi
   + P(X;\tau;\k) \bs (\partial_s\phi)\Big)*\tau*\tau
\\
 &= P(\partial_sX;X;\tau;\k) \bs (\nabla_s\phi + \phi)\,.
\end{align*}
Differentiating \eqref{EQnablasVcshort} with the help of \eqref{EQdiffrulebs} and \eqref{EQdiffrunebsprod} we
obtain
\begin{align}
\nabla_s^2V_c
 &=
   \IP{\c}{\tau} \nabla^3_s\k
 + P\Big(\tau;\L\tau;\cz;\fg;\df;\dkf{2}(\tau);\k\Big)
  \bs(\nabla_s^2\k + \nabla_s\k + \nabla_s\k\bs\nabla_s\k)
\notag\\
&\quad
 + P\Big(\tau;\L\tau;\M;\cz;\fg;\df;\dkf{2}(\tau);\dkf{3}(\tau,\tau);\dkf{2}(\k);\k\Big)
  \bs\nabla_s\k
\notag\\
&\quad
 + P\Big(\tau;\L\tau;\M;\cz;\fg;\df;\dkf{2}(\tau);\dkf{3}(\tau,\tau);\dkf{2}(\k);\k;\L\k\Big)
 + (\partial_sY)\tau\,,
\label{EQnablas2Vcstarproduct}
\end{align}
where $Y$ is as in \eqref{EQnablasVcshort}.
Using \eqref{EQnablas2Vcstarproduct}, Assumption \ref{Af}, and Lemmas
\ref{LMaprioriL2controlofK}--\ref{LMaprioriL2controlofNablaK}, we estimate
\begin{align}
 2\int_\gamma &\IP{\nabla_s^2\k}{\nabla_s^{4}V_c}\,ds
=
 2\int_\gamma \IP{\nabla_s^4\k}{\nabla_s^2V_c}\,ds
\notag\\
&\le
     c\int_\gamma |\nabla_s^4\k|\Big(|\nabla_s^3\k| + |\nabla_s^2\k| + |\nabla_s\k| + |\nabla_s\k|^2\Big)\,ds
   + cL(\gamma)
\notag\\
&\le
  \frac18\int_\gamma |\nabla_s^4\k|^2 ds
   + c\int_\gamma \Big(|\nabla_s^3\k|^2 + |\nabla_s^2\k|^2 + |\nabla_s\k|^4\Big)\,ds
   + c
\notag\\
&\le
  \frac14\int_\gamma |\nabla_s^4\k|^2 ds
   + c\,.
\label{EQprelimtermfornabla2k}
\end{align}
For the last inequality we used \eqref{EQstrongerDKSinterpfork} with
[$\nu=2$, $\mu=6$, $k=4$, $\sigma=\frac32<2$], 
[$\nu=2$, $\mu=4$, $k=4$, $\sigma=1<2$], and
[$\nu=4$, $\mu=4$, $k=4$, $\sigma=\frac54<2$].
It remains to estimate the term
\begin{align*}
2&\int_\gamma \nabla_s^2\k * P_{2,1}^2(\k;V_c)\,ds
  = 2\int_\gamma \nabla_s^2\k * \nabla_s^2\big(\k*\k*V_c\big)\,ds
\\
 &= 
    2\int_\gamma \nabla_s^2\k * \k*\k*\nabla_s^2V_c\,ds
  + 4\int_\gamma \nabla_s^2\k * \nabla_s\k*\k*\nabla_sV_c\,ds
  + \int_\gamma P_3^{4,2}(\k)*V_c\,ds.
\end{align*}
Integration by parts gives
\begin{align*}
    c\int_\gamma &\nabla_s^2\k * \k*\k*\nabla_s^2V_c\,ds
  + c\int_\gamma \nabla_s^2\k * \nabla_s\k*\k*\nabla_sV_c\,ds
\\
&= c\int_\gamma \big(
        \nabla_s^4\k * \k * \k
      + \nabla_s^3\k * \nabla_s\k * \k
      + \nabla_s^2\k * \nabla_s^2\k * \k
      + \nabla_s^2\k * \nabla_s\k * \nabla_s\k
               \big) * V_c\,ds
\\
&=
   c\int_\gamma \big( P_3^{4,3}(\k)
               \big) * V_c\,ds
 + c\int_\gamma 
        \nabla_s^4\k * \k * \k * V_c\, ds.
\end{align*}
Combining the last two equalities and estimating we have
\begin{align}
2&\int_\gamma \nabla_s^2\k * P_{2,1}^2(\k;V_c)\,ds
\notag\\
&\le
   c\int_\gamma (\P_3^{4,3} + \P_4^{5,3} + \P_6^{2,3}
                )(\k)\,ds
 + \frac1{8}\int_\gamma 
        |\nabla_s^4\k|^2\,ds
 + c\,.
\label{EQprelimtermfornabla2k2}
\end{align}
Interpolating using \eqref{EQstrongerDKSinterpfork} with
[$\nu=3$, $\mu=4$, $k=4$, $\sigma=\frac{9}{8}<2$], 
[$\nu=4$, $\mu=5$, $k=4$, $\sigma=\frac{3}{2}<2$], and
[$\nu=6$, $\mu=2$, $k=4$, $\sigma=\frac34<2$],
 we further estimate the first term by
\begin{align*}
   c\int_\gamma (\P_3^{4,3} + \P_4^{5,3} + \P_6^{2,3}
                )(\k)\,ds
\le
  \frac1{8}\int_\gamma |\nabla_s^{4}\k|^2ds
 + c,
\end{align*}
which, upon combination with \eqref{EQprelimtermfornabla2k}, \eqref{EQprelimtermfornabla2k2} and reinsertion
into \eqref{EQhigherorderevo1} (with $m=2$) yields
\begin{equation*}
\rd{}{t}\int_\gamma |\nabla_s^2\k|^2ds
+ \int_\gamma |\nabla_s^{4}\k|^2ds
 \le
  c\,.
\end{equation*}
Combining this estimate with \eqref{EQsimpleinterp} for $m=2$, $p=2$, $\varepsilon=1$, we obtain
\begin{equation}
\rd{}{t}\int_\gamma |\nabla_s^2\k|^2ds
 \le
 -\int_\gamma |\nabla_s^2\k|^2ds
 +  c\,.
\label{EQfinalestimateL2nablas2K}
\end{equation}
A contradiction argument completely analogous to that which gave \eqref{EQLMaprioriL2controlofNablaKCNCLSN}
now implies the bound \eqref{EQhigherorderboundsCNCLSN} for $m=2$, with $k_2 = c -
\vn{\nabla_s^2\k}_{L^2}^2\big|_{t=0}$, where $c$ is the constant from \eqref{EQfinalestimateL2nablas2K}
which depends only on $\SH(\gamma_0)$, $\lambda$, $n$, $|\L|$, $c_0$, $c_1$, $c_2$, and $c_3$.

From now on we shall assume $m\ge3$, and prove \eqref{EQhigherorderboundsCNCLSN} by induction.
The inductive hypothesis is that \eqref{EQhigherorderboundsCNCLSN} holds for all $1 \le p \le m-1$, and implies
\begin{equation}
P\big(\cz;\fg;\partial_s\gamma;\partial_s^2\gamma;\cdots;\partial_s^m\gamma;\L\partial_s\gamma;\cdots;\L\partial_s^m\gamma\big)
\le c,\text{ and } \int_\gamma |\partial_s^{m+1}\gamma|^2ds \le c\,.
\label{EQinduchyp}
\end{equation}
In the above and from now on (unless otherwise explicitly stated) $c$ shall denote a constant depending 
only on $\SH(\gamma_0)$, $\lambda$, $n$, $|\L|$, and $c_i$ for $i = 0,\ldots,m+1$.
We shall also employ the abbreviation
\begin{equation*}
\PX
 = P\Big(\cz;
         \fg;\partial_s\fg;\cdots;\partial_s^{m+1}\fg;
         \partial_s\gamma;\partial_s^2\gamma;\cdots;\partial_s^m\gamma;
         \L\partial_s\gamma;\cdots;\L\partial_s^m\gamma
\Big)\,.
\end{equation*}
The second inequality in \eqref{EQinduchyp} follows from the argument of Lemma 2.7 in \cite{DKS02} (see in
particular equation (2.20)), except here it is important for us to work in $L^2$ instead of $L^1$.
Clearly we also have $\vn{\nabla_s^{m-2}V_c}_{L^2}^2, \vn{\partial_s^{m-2}V_c}_{L^2}^2 \le c$ and
$\vn{\partial_s^pV_c}_{L^\infty} \le c$ for $1\le p \le m-3$.
For the remainder of the proof we shall apply these estimates and those given by Lemmas
\ref{LMaprioriL2controlofK}--\ref{LMaprioriL2controlofNablaK} typically without further comment.

The inductive hypothesis and Assumption \ref{Af} imply $|\PX|\le c$.
All components of $\PX$ are controlled by \eqref{EQinduchyp} apart from the many derivatives of $\fg$, which
we shall briefly discuss now.
Let us set some additional notation.
A partition of $J_k := \{1,\ldots,k\}$ is a family of pairwise disjoint non-empty subsets of $J_k$ whose union
is $J_k$.  The set of all functions from a partition $P$ of $J_k$ into $J_n$ is denoted by $P_n$, and the set
of all partitions of $J_k$ is denoted by $\mathbb{P}_k$.
A special case of \cite[Lemma 3]{M09} gives the following formula for the $q$-th derivative of $\fg =
(f\circ\gamma)$:
\begin{equation}
\partial_s^q(\fg)
 = \sum_{P\in\mathbb{P}_q} \sum_{\beta\in P_n} \bigg\{
     \Big( \prod_{B\in P} \pd{}{\gamma_{\lambda(B)}} \Big) \fg \bigg\}
     \bigg\{ \prod_{B\in P} \Big[ \Big( \prod_{b\in B} \pd{}{s} \Big) \gamma_{\lambda(B)} \Big] \bigg\}\,.
\label{EQchainrule}
\end{equation}
Taking absolute values and estimating, formula \eqref{EQchainrule} implies (cf. the proof of Corollary 12 in
\cite{M09}) that
\begin{equation*}
|\partial_s^q(\fg)| \le (1+q)^{n+q+1} A (1+B)^q\,,
\end{equation*}
where
\begin{equation*}
A = \max_{1\le|p|\le q} \Big| \frac{\partial^{|p|}\fg}{\partial \gamma^p} \Big|,
\quad\text{and}\quad
B = \max_{1\le i\le n}\max_{1\le|\beta|\le q} \Big| \frac{\partial^{|\beta|}\gamma}{\partial s} \Big|\,.
\end{equation*}
Clearly we have $A \le c$ for a constant $c$ depending only on $c_1,\ldots,c_q$ and $B \le
P(\partial_s\gamma;\cdots;\partial_s^q\gamma)$.
We conclude the estimate
\begin{equation}
|\partial_s^r(\fg)| \le c(1+q)^{n+r+1} 
|P(\partial_s\gamma;\cdots;\partial_s^r\gamma)|
 \le c\,.
\label{EQpartialsmfgest}
\end{equation}
Derivatives of $\fg$ up to and including the order $m+1$ are thus controlled in $L^\infty$ by combining
\eqref{EQpartialsmfgest} with \eqref{EQinduchyp} above.
This is enough to conclude $|\PX| \le c$.

Returning to the evolution equation \eqref{EQhigherorderevo1}, we employ integration by parts and the
induction hypothesis to estimate
\begin{align*}
\rd{}{t}&\int_\gamma |\nabla_s^m\k|^2ds
+ \frac32\int_\gamma |\nabla_s^{m+2}\k|^2ds
\\
 &\le
   2\int_\gamma \nabla_s^m\k * P_{2,1}^m(\k;V_c)\,ds
 + 2\int_\gamma \IP{\nabla_s^{m+2}\k}{\nabla_s^{m}V_c}\,ds
\\
 &\le
  \frac18\int_\gamma |\nabla_s^{m+2}\k|^2ds
 + c\int_\gamma \nabla_s^{m+2}\k * P_{2,1}^{m-2}(\k;V_c)\,ds
 + c\int_\gamma |\nabla_s^{m}V_c|^2ds
\\
 &\le
  \frac14\int_\gamma |\nabla_s^{m+2}\k|^2ds
 + c\int_\gamma |P_{2,1}^{m-2}(\k;V_c)|^2ds
 + c\int_\gamma |\nabla_s^{m}V_c|^2ds
\\
 &\le
  \frac14\int_\gamma |\nabla_s^{m+2}\k|^2ds
 + c\int_\gamma |\k|^4\,|\nabla_s^{m-2}V_c|^2\,ds
 + c\int_\gamma |P(\gamma;\cdots;\partial_s^m\gamma)|^2ds
\\
&\quad
 + c\int_\gamma |\nabla_s^{m}V_c|^2ds
\\
 &\le
  \frac14\int_\gamma |\nabla_s^{m+2}\k|^2ds
 + c\int_\gamma |\nabla_s^{m}V_c|^2ds + c\,.
\end{align*}
Therefore
\begin{align}
\rd{}{t}&\int_\gamma |\nabla_s^m\k|^2ds
+ \frac54\int_\gamma |\nabla_s^{m+2}\k|^2ds
 \le
  c\int_\gamma |\nabla_s^{m}V_c|^2ds + c\,,
\label{EQestfornablaMmaineq}
\end{align}
and it remains to estimate the term $\vn{\nabla_s^{m}V_c}_{L^2}^2$.
Clearly
\begin{align}
\int_\gamma |\nabla_s^{m}V_c|^2ds
&\le 
   c\int_\gamma |\nabla_s^{m+2}\cz\ts|^2ds
 + c\int_\gamma \Big|\nabla_s^{m}\big[(\L^T(\k-\c))^\bot\big]\Big|^2ds
\notag\\
&\quad
 + c\int_\gamma \Big|\nabla_s^{m}\big[|\c|^2\k\big]\Big|^2ds
 + c\int_\gamma \Big|\nabla_s^{m}\big[\IP{\c}{\tau}\nabla_s\k\big]\Big|^2ds
\notag\\
&\quad
 + c\int_\gamma \Big|\nabla_s^{m}\big[\IP{\cz}{\k}\k\big]\Big|^2ds
 + c\int_\gamma \Big|\nabla_s^{m}\big[\IP{\L\tau}{\tau}\k\big]\Big|^2ds
\notag\\
&\quad
 + c\int_\gamma \bigg|\nabla_s^{m}\Big[\df(\tau)\,\k\Big]\bigg|^2ds
 + c\int_\gamma \Big|\nabla_s^{m}\big[\fg\IP{\cz}{\tau}\k\big]\Big|^2ds
\notag\\
&\quad
 + c\int_\gamma \bigg|\nabla_s^{m}\Big[\df(\tau)\cz\ts\Big]^\bot\bigg|^2ds
 + c\int_\gamma \Big|\nabla_s^{m}\big[\fg\L\tau\big]^\bot\Big|^2ds
\notag\\
&\quad
 + c\int_\gamma \Big|\nabla_s^{m}\big[\IP{\c}{\tau}(\df)^T\big]^\bot\Big|^2ds\,.
\label{EQestfornablaMVc1}
\end{align}
We shall deal with each term of \eqref{EQestfornablaMVc1} in turn.
The general idea to keep in mind is that only the terms with the highest number of derivatives of $\k$ need to
be interpolated explicitly, using the induction hypothesis and our earlier estimates to deal with any other
auxilliary contributions.
Indeed, the highest order contribution in \eqref{EQestfornablaMVc1} above is in the fourth term, where we have
$\vn{\IP{\c}{\tau}\nabla_s^{m+1}\k}_{L^2}^2$.
If this were one order higher, then our interpolation method (using \eqref{EQDKSinterpfork} or
\eqref{EQstrongerDKSinterpfork}) would fail.
As it stands however, we are able to interpolate this term without difficulty (see \eqref{EQestfornablaMVc6}
below).
Apart from terms with a large number of derivatives of curvature, one must also be wary of terms with a high
degree of curvature.
The term with the highest degree of curvature above is the fifth, where among other lower-order contributions
we must deal with $\int_\gamma \cz*\cz*P_4^{2m,m}(\k)\,ds$.
This is far from critical for the interpolation inequality however, which could handle terms with $2m$
derivatives distributed among eight copies of $\k$, that is, terms of the form
$\int_\gamma P_8^{2m,m+1}(\k)\,ds$.

Let us begin with the first term.
A straightforward computation yields
\begin{align*}
\nabla_s^{m+2}\cz
 &=
   \nabla_s^{m-2}\big[ (\L\nabla_s^2\k)^\bot - \IP{\tau}{\L\tau}\nabla_s^2\k \big]
 - 3\IP{\nabla_s^{m-1}\k}{\k}(\L\tau)^\bot
 - 3\IP{\tau}{\L\partial_s^{m-2}\nabla_s\k}\k
\\
&\quad
 - \IP{\partial_s^{m-2}\nabla_s\k}{\L\tau}\k
 + P(\cz;\fg;\partial_s\gamma;\partial_s^2\gamma)*\nabla_s^{m-1}\k
 + \PX
\,,
\end{align*}
which one squares and integrates to find
\begin{align*}
c\int_\gamma |\nabla_s^{m+2}\cz\ts|^2ds
 &\le
   \int_\gamma \Big|\big(
                 \nabla_s^{m-2}\big[ (\L\nabla_s^2\k)^\bot - \IP{\tau}{\L\tau}\nabla_s^2\k \big]
                    \big)^\bot\Big|^2ds
\\
&\quad
 + c\int_\gamma |P(\cz;\fg;\partial_s\gamma;\partial_s^2\gamma)|^2|\partial_s^{m+1}\gamma|^2ds
 + c\int_\gamma |\PX|^2ds
\\
 &\le
   \int_\gamma \Big|\big(
                 \nabla_s^{m-2}\big[ (\L\nabla_s^2\k)^\bot - \IP{\tau}{\L\tau}\nabla_s^2\k \big]
                    \big)^\bot\Big|^2ds
 + c\,.
\end{align*}
We expand the first term on the right with
\begin{align}
\nabla_s^{m-2}&\big[ (\L\nabla_s^2\k)^\bot - \IP{\tau}{\L\tau}\nabla_s^2\k \big]
\notag\\
 &=
   \nabla_s^{m-2}\big(\L\nabla_s^2\k\big)
 - \nabla_s^{m-3}\big(\IP{\tau}{\L\nabla_s^2\k}\k\big)
 - \nabla_s^{m-2}\big(\IP{\tau}{\L\tau}\nabla_s^2\k\big)
\,.
\label{EQestfornablaMexp0}
\end{align}
Interchanging $\nabla_s$ with $\L$, we estimate
\begin{align}
\int_\gamma |\nabla_s^{m-2}(\L\nabla_s^2\k)|^2 ds
 &\le
   c\int_\gamma |\nabla_s^m\k|^2ds
 + c\int_\gamma |P_4^{2m-2,m-1}(\k)|\,ds
\notag\\
&\quad
 + c\int_\gamma |\PX|^2\big(1 + |\partial_s^{m+1}\gamma|^2\big)\, ds
\notag\\
&\le \frac{1}{42}\int_\gamma |\nabla_s^{m+2}\k|^2ds + c\,,
\label{EQestfornablaMVc2}
\end{align}
where we again employed the inductive hypothesis and the interpolation inequality \eqref{EQDKSinterpfork} with
[$\nu=4$, $\mu=2m-2$, $k=m+2$, $\sigma=\frac{2m-1}{m+2}<2$].
Keeping in mind $m\ge3$, for the second and third terms of \eqref{EQestfornablaMexp0} we estimate
\begin{align*}
\int_\gamma &\Big|\nabla_s^{m-3}\big(\IP{\tau}{\L\nabla_s^2\k}\k\big)\Big|^2ds
+ \int_\gamma \Big|\nabla_s^{m-2}\big(\IP{\tau}{\L\tau}\nabla_s^2\k\big)\Big|^2ds
\\
&\le
   c\int_\gamma |\nabla_s^m\k|^2ds
 + c\int_\gamma |\PX|^2\big(1 + |\partial_s^{m+1}\gamma|^2\big)\,ds
\\
&\le \frac{1}{42}\int_\gamma |\nabla_s^{m+2}\k|^2ds + c\,.
\end{align*}
Combined with \eqref{EQestfornablaMVc2} this gives the desired estimate for the first term on the right hand
side of \eqref{EQestfornablaMVc1}:
\begin{equation}
\label{EQestfornablaMVc3}
\int_\gamma \big|\nabla_s^{m+2}\cz\ts\big|^2ds
\le \frac{1}{42}\int_\gamma |\nabla_s^{m+2}\k|^2ds + c\,.
\end{equation}

We continue by estimating the second term in \eqref{EQestfornablaMVc1} with
\begin{align}
c\int_\gamma &\Big|\nabla_s^{m}\big[(\L^T(\k-\c))^\bot\big]\Big|^2ds
\notag\\
&=
  c\int_\gamma \Big|\nabla_s^{m}\big[\L^T(\k-\c) - \IP{(\L^T(\k-\c))}{\tau}\tau\big]\Big|^2ds
\notag\\
&=
  c\int_\gamma \Big|\nabla_s^{m}\big[\L^T\k - \IP{\L^T\k\L}{\tau}\tau\big]         
                  + \nabla_s^{m}\big[-\L^T\c + \IP{\L^T\c}{\tau}\tau\big]\Big|^2ds
\notag\\
&\le
  c\int_\gamma \Big|\nabla_s^{m}\big[\L^T\k - \IP{\L^T\k}{\tau}\tau\big]\Big|^2ds
+ c\int_\gamma |\PX|^2\big(1+ |\partial_s^{m+1}\gamma|^2\big)\,ds
\notag\\
&\le
  c\int_\gamma \Big|\nabla_s^{m}\big[\L^T\k\big]\Big|^2ds
+ c\int_\gamma \Big|\nabla_s^{m-1}\big[\IP{\L^T\k}{\tau}\k\big]\Big|^2ds
+ c
\notag\\
&\le
  c\int_\gamma |\nabla_s^{m}\k|^2ds
+ c\int_\gamma \big|P_2^{2m-2,m-1}(\k)\big|^2ds
+ c\int_\gamma |\PX|^2\big(1 + |\partial_s^{m+1}\gamma|^2\big)\,ds
+ c
\notag\\
&\le
  \frac1{42}\int_\gamma |\nabla_s^{m+2}\k|^2ds + c\,.
\label{EQestfornablaMVc4}
\end{align}
In obtaining the last inequality we used \eqref{EQDKSinterpfork} to estimate the second term $\int_\gamma
\big|P_2^{2m-2,m-1}(\k)\big|^2ds$ with 
[$\nu=2$, $\mu=2m-2$, $k=m$, $\sigma=\frac{2m-2}{m}<2$].
Similarly, the third term in \eqref{EQestfornablaMVc1} is estimated by
\begin{align}
c\int_\gamma \Big|\nabla_s^{m}\big[|\c|^2\k\big]\Big|^2ds
&\le
c\int_\gamma |\nabla_s^{m}\k|^2ds
+ c\int_\gamma |\PX|^2\big(1 + |\partial_s^{m+1}\gamma|^2\big)\,ds
\notag\\
&\le
  \frac1{42}\int_\gamma |\nabla_s^{m+2}\k|^2ds + c\,.
\label{EQestfornablaMVc5}
\end{align}
The fourth term in \eqref{EQestfornablaMVc1} is one order higher than the others.
Nevertheless, we may estimate it in an analogous manner:
\begin{align}
c\int_\gamma &\Big|\nabla_s^{m}\big[\IP{\c}{\tau}\nabla_s\k\big]\Big|^2ds
\notag\\
&=
  c\int_\gamma \big|\nabla_s^{m-1}[
                 \IP{\c}{\tau}\nabla_s^2\k
               + \IP{\L\tau}{\tau}\nabla_s\k
               + \df(\tau)\nabla_s\k
               + \IP{\c}{\k}\nabla_s\k
               ]\big|^2ds
\notag\\
&\le
  c\int_\gamma \Big|\nabla_s^{m-2}\big[
                 \IP{\c}{\tau}\nabla_s^3\k
               + 2\df(\tau)\nabla_s^2\k
               + 2\IP{\L\tau}{\tau}\nabla_s^2\k
               + 2\IP{\c}{\k}\nabla_s^2\k
\notag\\
&\qquad\qquad\qquad
               + \df(\k)\nabla_s\k
               + \dkf{2}(\tau,\tau)\nabla_s\k
               + 2\IP{\L\tau}{\k}\nabla_s\k
               + \IP{\L\k}{\tau}\nabla_s\k
\notag\\
&\qquad\qquad\qquad
               + \IP{\c}{\partial_s\k}\nabla_s\k
               + \fg\nabla_s\k
               \big]\Big|^2ds
\notag\\
&\le
  c\int_\gamma \Big|\nabla_s^{m-2}\big[
                 \IP{\c}{\tau}\nabla_s^3\k
               + 2\df(\tau)\nabla_s^2\k
               + 2\IP{\L\tau}{\tau}\nabla_s^2\k
               + 2\IP{\c}{\k}\nabla_s^2\k
               \big]\Big|^2ds
\notag\\*
&\quad
+ c\int_\gamma |\PX|^2\big(1 + |\partial_s^{m+1}\gamma|^2\big)\,ds
\notag\\
&\le
  c\int_\gamma \Big|\nabla_s^{m-3}\big[
                 \IP{\c}{\tau}\nabla_s^4\k
               + 3\df(\tau)\nabla_s^3\k
               + 3\IP{\L\tau}{\tau}\nabla_s^3\k
               + 3\IP{\c}{\k}\nabla_s^3\k
               \big]\Big|^2ds
\notag\\*
&\quad
+ c\int_\gamma |\PX|^2\big(1 + |\partial_s^{m+1}\gamma|^2\big)\,ds
\notag\\
&\le
  c\int_\gamma \Big|\nabla_s^{m-3}\big[
                 \IP{\c}{\tau}\nabla_s^4\k
               \big]\Big|^2ds
 + c\int_\gamma |\nabla_s^m\k|^2ds
\notag\\*
&\quad
 + c\int_\gamma |\PX|^2\big(1 + |\partial_s^{m+1}\gamma|^2\big)\,ds
\notag\\
&\le
  c\int_\gamma |\nabla_s^{m+1}\k|^2ds
 + c\int_\gamma |\nabla_s^m\k|^2ds
 + c\int_\gamma |\PX|^2\big(1 + |\partial_s^{m+1}\gamma|^2\big)\,ds
\notag\\
&\le
  \frac1{42}\int_\gamma |\nabla_s^{m+2}\k|^2ds
 + c\,.
\label{EQestfornablaMVc6}
\end{align}
For the last inequality we used \eqref{EQDKSinterpfork} with
[$\nu=2$, $\mu=2m+2$, $k=m+2$, $\sigma = \frac{2m+2}{m+2}<2$] and
[$\nu=2$, $\mu=2m$, $k=m+2$, $\sigma = \frac{2m}{m+2}<2$].

The remaining seven terms in \eqref{EQestfornablaMVc1} are estimated in a similar manner, using a combination
of the methods used for the first four terms.
Briefly,
\begin{align}
  c\int_\gamma &\Big|\nabla_s^{m}\big[\IP{\c}{\k}\k\big]\Big|^2ds
+ c\int_\gamma \Big|\nabla_s^{m}\big[\IP{\L\tau}{\tau}\k\big]\Big|^2ds
\notag\\
+\, c\int_\gamma &\Big|\nabla_s^{m}\big[\fg\IP{\cz}{\tau}\k\big]\Big|^2ds
+ c\int_\gamma \Big|\nabla_s^{m}\big[\df(\tau)\,\k\big]\Big|^2ds
\notag\\
&\le
 c\int_\gamma |\nabla_s^{m}\k|^2ds
+ c\int_\gamma |\PX|^2\big(1 + |\partial_s^{m+1}\gamma|^2\big)\,ds
\notag\\
&\le
  \frac1{42}\int_\gamma |\nabla_s^{m+2}\k|^2ds + c\,,\ \text{and}
\notag\\
c\int_\gamma &\Big|\nabla_s^{m}\big[\df(\tau)\,\cz\ts\big]^\bot\Big|^2ds
+ c\int_\gamma \Big|\nabla_s^{m}\big[\fg\L\tau\big]^\bot\Big|^2ds
+ c\int_\gamma \Big|\nabla_s^{m}\big[\IP{\c}{\tau}\df\big]^\bot\Big|^2ds
\notag\\
&\le
  c\int_\gamma |\PX|^2\big(1 + |\partial_s^{m+1}\gamma|^2\big)\,ds
 \le
  c\,.
\label{EQestfornablaMVc7}
\end{align}

Inserting the estimates \eqref{EQestfornablaMVc3}, \eqref{EQestfornablaMVc4}, \eqref{EQestfornablaMVc5},
\eqref{EQestfornablaMVc6}, \eqref{EQestfornablaMVc7} into \eqref{EQestfornablaMVc1} yields
\begin{align*}
\int_\gamma |\nabla_s^{m}V_c|^2ds
&\le 
  \frac1{4}\int_\gamma |\nabla_s^{m+2}\k|^2ds
 + c\,,
\end{align*}
which we insert into \eqref{EQestfornablaMmaineq} to find
\begin{equation*}
\rd{}{t}\int_\gamma |\nabla_s^m\k|^2ds
+ \int_\gamma |\nabla_s^{m+2}\k|^2ds
 \le
  c\,,
\end{equation*}
which, upon combination with \eqref{EQsimpleinterp} for $p=2$, $\varepsilon=1$, yields
\begin{equation*}
\rd{}{t}\int_\gamma |\nabla_s^m\k|^2ds
 \le
  c
- \int_\gamma |\nabla_s^{m}\k|^2ds\,.
\end{equation*}
The above estimate, by an argument identical to that for the cases $m=1$ and $m=2$ treated earlier, yields the bound
\eqref{EQhigherorderboundsCNCLSN} for all $m$.
\end{proof}


\begin{proof}[Proof of Theorem \ref{TMmt}]
Given our previous estimates (in particular Lemma \ref{LMhigherorderbounds}), this follows in a manner similar
to that of \cite[Theorem 3.2]{DKS02}.
For the convenience of the reader we reproduce the argument with the necessary modifications here.

Let us first note that Lemma \ref{LMaprioricontrolofL} gives uniform upper and lower bounds of $L(\gamma)$
and combining Lemma \ref{LMhigherorderbounds} with \cite[Lemma 2.7]{DKS02} gives $L^\infty$ control of all
derivatives of curvature; summarising, we have
\begin{equation}
 \frac{2\pi^2}{\SH(\gamma_0)}
\ \le\ L(\gamma)\ \le\ \frac{1}{\lambda}\SH(\gamma_0)\,,\quad\text{and}\quad
\vn{\partial_s^{m+2}\gamma}_{L^\infty} \le c(m)\,,
\label{EQbounds}
\end{equation}
where $m$ is a non-negative integer and $c(m)$ is a constant depending only on $m$, $\SH(\gamma_0)$,
$\lambda$, $n$, $|\L|$, and $c_i$ for $i = 0,\ldots,m+1$.

Suppose $T<\infty$. Our goal is to convert the bounds \eqref{EQbounds} to the following:
\begin{equation}
\vn{\partial_u^{k}\gamma}_{L^\infty} \le \tilde{c}(k)\,,
\label{EQdesiredbounds}
\end{equation}
where $u$ is the parameter from the original parametrisation of $\gamma$ (recall that we reparametrised by
arclength via $s(u) = \int_0^u |\partial_u\gamma|\,du$) and $k$ is a non-negative integer.  We wish to prove
\eqref{EQdesiredbounds} for $t\in(0,T)$ with constants $\tilde{c}(k)$ depending only on $k$, $\SH(\gamma_0)$,
$\lambda$, $n$, $|\L|$, $T$ and $c_i$ for $i = 0,\ldots,k+1$.
This will imply that we can extend $\gamma$ smoothly to $T$, and beyond by short time existence, contradicting
the finite maximality of $T$.

To begin, note that $ds = |\partial_u\gamma|\,du$ satisfies $\partial_t\,ds = -\IP{\k}{V}
|\partial_u\gamma|\,du$, so that in combination with \eqref{EQbounds} and Lemmas \ref{LMaprioricontrolofC0},
\ref{LMaprioricontrolofC}, we have
\begin{equation}
\frac1c \le |\partial_u\gamma| \le c\,,
\label{EQinduction0}
\end{equation}
with $c$ depending only on $\SH(\gamma_0)$, $\lambda$, $n$, $|\L|$, $T$, and $c_i$ for $i = 0,\ldots,3$.
Let $h:\S^1\rightarrow\R$ be a smooth function.  The general interchange formula
\begin{equation*}
\partial_u^kh - |\partial_u\gamma|^k\partial_s^kh
 = P(|\partial_u\gamma| ; \cdots ; \partial_u^{k-1}|\partial_u\gamma|
     ; h ; \cdots ; \partial_s^{k-1}h)
\end{equation*}
with $h = \IP{\k}{V}$ implies
\begin{align}
\vn{\partial_u^k\IP{\k}{V}}_{L^\infty}
 &\le |\partial_u\gamma|^kP(\cz;\fg;\partial_s\gamma;\cdots;\partial_s^{k+3}\gamma)
\notag\\
&\quad + P(|\partial_u\gamma| ; \cdots ; \partial_u^{k-1}|\partial_u\gamma|
     ; \cz;\fg;\partial_s\gamma;\cdots;\partial_s^{k+2}\gamma)\,.
\label{EQinterchange1}
\end{align}
Let us prove
\begin{equation}
\vn{(\partial_u^{k}|\partial_u\gamma|)}_{L^\infty} \le \tilde{c}(k)\,,
\label{EQinduction1}
\end{equation}
where $k$ is a non-negative integer, by induction.
Estimate \eqref{EQinduction0} clearly implies \eqref{EQinduction1} for $k=0$, and assuming
\eqref{EQinduction1} holds for $k=0,\ldots,p-1$, we combine \eqref{EQbounds} with \eqref{EQinterchange1} and the
evolution of $|\partial_u\gamma|$ to obtain
\begin{align*}
\partial_t(\partial_u^p|\partial_u\gamma|) + \IP{\k}{V}(\partial_u^p|\partial_u\gamma|)
 &= P\big(\IP{\k}{V};\cdots;\partial_u^{p}\IP{\k}{V} ;
     |\partial_u\gamma| ; \cdots ; \partial_u^{p-1}|\partial_u\gamma|\big)
\\
 &\le c\,,
\end{align*}
which by a simple Gronwall argument implies \eqref{EQinduction1} for $k=p$.  We have thus proven
\eqref{EQinduction1} for all $k$.

To see that \eqref{EQinduction1} implies \eqref{EQbounds}, note that
$|\partial_s\gamma| = 1$,
\begin{align*}
\partial_u^k\gamma
 &= \partial_u^{k-1}\big(|\partial_u\gamma|\partial_s\gamma\big)
 = P(\partial_s\gamma ; \cdots ; \partial_s^k\gamma ; |\partial_u\gamma| ; \cdots ;
\partial_u^{p-1}|\partial_u\gamma|\big)\,,
\end{align*}
and use \eqref{EQbounds}.
This finishes the proof of \eqref{EQdesiredbounds} for $k\ge1$.
For $k=0$ we combine \eqref{EQflow} with \eqref{EQbounds} to obtain ($\{e_i\}_{i=1}^n$ is the standard basis
of $\R^n$)
\begin{equation*}
\IP{\gamma}{e_i} = \int_0^t \IP{V}{e_i}\,dt \le \tilde{c}(2)\,t\,,
\end{equation*}
which clearly implies \eqref{EQdesiredbounds} for $k=0$.

We have therefore shown $T=\infty$.
In order to obtain the convergence statement, first note that the estimates provided by Lemma
\ref{LMhigherorderbounds} are uniform, and so do not degenerate as $t\rightarrow\infty$.
These estimates are in the arc-length parametrisation of $\gamma$ however, and the estimates obtained in the
contradiction argument above are not uniform in time.
So, let us reparametrise $\gamma$ at each time such that it remains parametrised by arc-length.
By the estimates \eqref{EQbounds} we have every derivative of $\gamma$ bounded a-priori.
Lemma \ref{LMaprioricontrolofL} implies that 
Composing the flow with a sequence of translations $p_j\in\R^n$ ($p_j = \gamma(0,t_j)$ is an
allowable choice) allows us to bound the length of $\gamma$.
We therefore conclude that there exists a sequence of times $t_j\rightarrow\infty$ such that the
subsequence $\gamma_{t_j}-p_j$ converges as $j\rightarrow\infty$ to a smooth limit curve
$\gamma_\infty$.
Now by Lemma \ref{LMfirstvariation}, we have that
\begin{equation*}
\int_0^\infty\int_\gamma |\BH(\gamma)|^2 ds\,dt = \SH(\gamma_0) - \SH(\gamma_\infty) \le \SH(\gamma_0)\,,
\end{equation*}
which shows that $\vn{\BH(\gamma)}_2^2 \in L^1([0,\infty))$.
The bounds \eqref{EQbounds} imply that $\partial_t \vn{\BH(\gamma)}_2^2$ is uniformly bounded.
Therefore up to the choice of another subsequence (which we also denote by $t_j$) we have
$\BH(\gamma_{t_j})\rightarrow 0$ as $j\rightarrow\infty$.
Lemma \ref{LMaprioricontrolofC0} implies that up to yet another choice of subsequence (again denoted by $t_j$)
the limit $\vec{c}_\infty = \lim_{j\rightarrow\infty} \big(\ca\circ\gamma(\cdot,t_j) +
(f\circ\gamma(\cdot,t_j))\tau(\cdot)\big)$ exists.
Therefore the functional $\SHlim$ and its Euler-Lagrange operator $\BHlim$ exist, and by the above argument
the limiting curve $\gamma_\infty(s) = \lim_{j\rightarrow\infty}\big(\gamma(s,t_j)-p_j\big)$ satisfies
$\BHlim(\gamma_\infty) = 0$.
\end{proof}

\begin{proof}[Proof of Theorem \ref{TMmt2}]
In either of the cases where $\L$ is invertible and non-vanishing, or the pair $(f,\gamma_0)$ satisfy Assumption
\ref{ASfproperlike}, we are able to bound $\gamma$ uniformly in $L^\infty$ (see Lemmas
\ref{LMaprioricontrolofGamma} and \ref{LMaprioricontrolofGammaf}), thus restricting the flow to a ball
$B_\rho(0)\subset\R^n$, and removing the need for the translations $p_j$.
We additionally recover convergence of the full sequence in this case, as the following basic argument shows.
Let $T^* \in (0,\infty)$.  Suppose a pair of subsequences $\gamma_{t_j}$, $t_j \rightarrow T^*$, and
$\gamma_{s_j}$, $s_j \rightarrow T^*$, converge to distinct limits $\hat\gamma$ and $\tilde\gamma$.  Then 
for some $e_k$ in the standard basis of $\R^n$ we have $\IP{\hat\gamma}{e_k} \ne \IP{\tilde\gamma}{e_k}$.
This is in contradiction with
\begin{equation*}
\Big|\IP{\gamma_{t_j}}{e_k} - \IP{\gamma_{s_j}}{e_k}\Big|
 = \Big|\int_{s_j}^{t_j} \partial_t\IP{\gamma}{e_k}\,dt \Big|
 \le c|t_j - s_j|\,.
\end{equation*}
Therefore the full sequence $\gamma_t$ converges on $(0,P)$ for each $P$, and taking $P\rightarrow\infty$
gives the convergence result on $(0,\infty)$.  Carrying out the argument for each of the derivatives of
$\gamma$ shows that this gives smooth convergence.

Let us consider now the case where $f$ and $\ca$ are constant.
Setting $p(t) = \gamma(0,t)$, recall than an allowable choice for the sequence of translations $p_j$ giving
the earlier subconvergence is $p_j = p(t_j)$, since then Lemma \ref{LMaprioricontrolofL} would give $|\gamma|
\le 2L(\gamma) \le 2\lambda^{-1}\SH(\gamma_0)$.
Denote by $N$ the normal bundle over $\gamma_\infty$, which has as elements of the fibre at $s$
vectors in $\R^n$ which are normal to $\gamma_\infty$ at $\gamma_\infty(s)$.
Any curve close to $\gamma_\infty$ in $C^m$ ($m$ large enough) can be written uniquely as graphs
over $\gamma_\infty$.
This gives us a chart for the space of curves near $\gamma_\infty$, which takes a neighbourhood $U$
about the origin in the space of $C^m$ sections of $N$ to a neighbourhood $Q$ of $\gamma_\infty$ in
the space of $C^m$ curves, given by
\[
g\in\Gamma(N)\mapsto \{\gamma_g(s) = \gamma_\infty(s)+g(s)\}\,.
\]
Intersecting this with $C^m$ gives a correspondence with a neighbourhoood of $\gamma_\infty$ in the
space of $C^m$ curves.
Furthermore, there is a constant $C$ such that the arc-length parameter on a curve $\gamma_g$ for
$g\in U$ is equivalent to the arc-length parameter on $\gamma_\infty$:
\[
\frac1C\,ds
 \le ds_g
 = \sqrt{(1+\IP{\k}{g})^2 + |dg|^2}\,ds
 \le C\,ds\,.
\]
Let us use $ds$ to denote the arc-length element along $\gamma_\infty$.
On the set $Q$, composition with this chart makes the energy $\SH$ into an analytic functional $\SJ$
on $V$ (assuming $m$ is large enough).
Our flow is not the $L^2(ds)$-gradient flow of $\SJ$, but we can control the extent to which it
fails to be: The angle between the flow of $g$ and the negative gradient vector of $\SJ$ is bounded
away from $\pi/2$.
We have
\begin{align}
\frac{d}{dt}\SJ(g+t\phi)\Big|_{t=0}
 &= \frac{d}{dt}\SH(\gamma_{g+t\phi})\Big|_{t=0}
\notag\\
 &= \IP{\BH(\gamma_g)}{\phi}_{L^2(ds_g)}
\label{EQgfgp}\\
 &= \IP{\BH(\gamma_g)\sqrt{(1+\IP{\k}{g})^2 + |dg|^2}}{\phi}_{L^2(ds)}\,.
\notag
\end{align}
The gradient vector $\BJ$ at $g$ is a section of $N$, given by
\[
\BJ(g)(s) = \pi_1\BH(\gamma_g)(s)\sqrt{(1+\IP{\k}{g})^2 + |dg|^2}\,,
\]
where $\pi_1: N_{\gamma_g}(s)\rightarrow N_{\gamma_\infty}(s)$ is the orthogonal projection onto the
normal space $N_{\gamma_\infty}(s)$.  Since $\SH$ is invariant under reparametrisation, we know that
$\BH(\gamma_g)$ is a section of the
normal bundle of $\gamma_g$.
The gradient flow in the graphical parametrisation (c.f. \eqref{EQgfgp}) is given by
\[
\partial_tg = -\pi_1\BH(\gamma_g)
\,.
\]
The angle between the normal spaces $N_{\gamma_g}(s)$ and $N_{\gamma_\infty}(s)$ is well-controlled,
and it follows that the $L^2(ds)$ norms of both $\partial_tg$ and $\BJ(g)$ are comparable to the
$L^2(ds)$ norm of $\BH(\gamma_g)$.
It follows that for $m$ sufficiently large (i.e., for $j$ sufficiently large) we have that the angle
between $\partial_tg$ and $\BJ(g)$ is bounded away from $\pi/2$.
That is, there exists a $c_0>0$ such that
\[
\IP{\partial_tg}{\BJ(g)}_{L^2(ds)}
\ge c_0\vn{\partial_tg}_{L^2(ds)}\vn{\BJ(g)}_{L^2(ds)}\,.
\]

Keeping this in mind, we now follow an idea of Simon \cite{S83} for the functional $\SJ$.
As we are evolving by a gradient flow of an analytic functional, we may employ the Lyapunov-Schmidt reduction
and the classical Lojasiewicz inequality (see \cite[Proof of Theorem 3]{S83}) to obtain that in a
neighbourhood $U$ of $\gamma_\infty$ in $C^{k,\alpha}$ for some large $k$ that
\begin{equation*}
\vn{\BJ(g(\cdot,t))}_2 \ge |\SJ(g(\cdot,t)) - \SJ(g_\infty)|^\alpha
\end{equation*}
for some $\alpha<1$.
Note of course that $\SJ(g_\infty) = \SHlim(\gamma_\infty)$ and we identify $g_\infty =
\gamma_\infty$.
In particular, while the solution $g(\cdot,t)$ remains in $U$ we have
\begin{align*}
\frac{d}{d t}&(\SJ(g(\cdot,t)-\SJ(g_\infty))^{1-\alpha}
\\
  &= -(1-\alpha)(\SJ(g(\cdot,t)-\SJ(g_\infty))^{-\alpha}
      \IP{\partial_tg}{\BJ(g)}_{L^2(ds_g)}
\\
  &\leq -(1-\alpha)
      \IP{\partial_tg}{\frac{\BJ(g)}{\vn{\BJ(g(\cdot,t))}_{L^2(ds_g)}}}_{L^2(ds_g)}
\\
  &\leq -c(1-\alpha)\vn{\partial_t g(\cdot,t)}_{L^2(ds_g)}\,,
\end{align*}
from which it follows (see \cite[Proof of Lemma 1]{S83}) that
\begin{equation}
\label{EQconv}
\vn{g(\cdot,t_j+t)-g(\cdot,t_j)}_2
  \leq \frac{c}{1-\alpha}|\SJ(g(\cdot,t_j))-\SJ(g_\infty)|^{1-\alpha}
\end{equation}
as long as $g(\cdot,t_j+t)$ remains within $U$.
Now the translation invariance allows us to enact the above argument completely analogously for
a neighbourhood $U+p_j$ of $g_\infty + p_j = \gamma_\infty+p_j$, and so the estimate \eqref{EQconv}
holds so long as $g$ remains within any of the translated neighbourhoods $U+p_j$.
The right hand side of \eqref{EQconv} can be made as small as desired by choosing $j$ large.
It follows that for $j$ sufficiently large we have $g(\cdot,t_j+t) \in U + p_j$ for all $t>0$.
It follows in particular that the solution does not escape to infinity, and converges to one of the translated
stationary solutions.
Since the flow $g$ is $\gamma$ in the graphical parametrisation, the same conclusion holds for
$\gamma$.

This proves the full convergence of the flow in each of the cases $(i)$--$(iii)$.
The criticality of the limit is an obvious consequence of the argument used to obtain the
criticality of $\gamma_\infty$ in the proof of Theorem \ref{TMmt}.
Note that the limiting functional is unique and there is no need to take any subsequences.
\end{proof}

%

\appendix
\section*{Appendix}

\begin{proof}[Proof of Lemma \ref{LMcirclesclassification}]
We shall determine for which radii $S_\rho$ is critical for $\SH$ in the cases where $f = \text{const.}$
and
\begin{itemize}
\item $\ca$ is a translation, or
\item $\ca$ is a rotation satisfying $\ca(S_\rho(\S^1)) = S_\rho(\S^1)$.
\end{itemize}
We write $\cz = \ca S_\rho = \L S_\rho + \M$ where $\L$ is a matrix of rotation through an angle of
$\theta$ about an axis $e_i$ in $\R^n$ with $\L S_\rho(\S^1) = S_\rho(\S^1)$.
Any circle centred at the origin satisfies $\k = -\rho^{-2}S_\rho$, and so $\nabla_s^m\k = 0$ for all $m\ge1$.
We begin by computing
\begin{align*}
\BH(S_\rho)
&=
  - \L\k
  + \IP{\L\k}{\tau}\tau
  + \frac12|\k|^2\k
  - \IP{\cz}{\k}\k
  - \lambda\k
  - \frac12|\cz+f\tau|^2\k
\\&\quad
  - \L^T(\k-\cz - f\tau)
  - f\L\tau
  + \IP{\k}{\L\tau}\tau
  - \IP{\cz}{\L\tau}\tau
  + f\IP{\cz}{\tau}\k
\\
&=
  - \L\k
  - (\L)^T\k
  + \IP{\L\k}{\tau}\tau
  + \frac12|\k|^2\k
  - \lambda\k
  - \IP{\cz}{\k}\k
\\&\quad
  - \frac12|\cz\ts|^2\k
  - \frac12|f|^2\k
  + (\L)^T\cz
  + f(\L)^T\tau
  - f\L\tau
  + \IP{\k}{\L\tau}\tau
  - \IP{\cz}{\L\tau}\tau
  \,.
\end{align*} 
In the first case we have $\L = 0$.
This implies
\begin{align*}
\BH(S_\rho)
 &= 
    \frac12|\k|^2\k
  - \IP{\M}{\k}\k
  - \lambda\k
  - \frac12|\M|^2\k
  - \frac12|f|^2\k
\\
 &= 
   \frac12\k\,\Big(|\k|^2
  - 2\IP{\M}{\k}
  - 2\lambda
  - |\M|^2
  - |f|^2\Big)
  \,,
\end{align*}
which vanishes only if $\M = 0$.
Indeed one finds that the solution must be the circle $S_\rho$ with radius satisfying
\begin{equation*}
\frac1\rho = \sqrt{2\lambda + |f|^2}\,.
\end{equation*}
Now let us consider the second case where $\M = 0$, which gives
\begin{align*}
\IP{\BH(S_\rho)}{\k}
  &=
  - 2\IP{\k}{\L\k}
  + \frac12|\k|^4
  - \lambda|\k|^2
  - \IP{\cz}{\k}|\k|^2
  - \frac12|\cz\ts|^2|\k|^2
\\&\quad
  - \frac12|f|^2|\k|^2
  + \IP{\cz}{\L\k}
  + f\IP{\tau}{\L\k}
  - f\IP{\k}{\L\tau}
  \,.
\end{align*}
There are two possibilities.  Either $\L$ is a rotation of the plane $\text{span}\{e_1,e_2\}$ leaving the
$(n-2)$ subspace $\text{span}\{e_3,\ldots,e_n\}$ invariant, or $\L$ is a rotation in a plane
$\text{span}\{e_i,e_j\}$, $i\ne 1,2$, $j\ne 1,2$, $i\ne j$, which leaves the plane $\text{span}\{e_1,e_2\}$
invariant.
Let us consider the second case first, whence the above equation becomes
\begin{align*}
\IP{\BH(S_\rho)}{\k}
  &=
  - 2|\k|^2
  + \frac12|\k|^4
  - \lambda|\k|^2
  + \rho^2|\k|^4
  - \frac12\rho^4|\k|^4
  - \frac12|f|^2|\k|^2
  - \rho^2|\k|^2
\\
&=
 \frac12|\k|^4\Big(
  - 3\rho^4
  - \rho^2\big(2
   + 2\lambda
   + |f|^2\big)
  + 1
 \Big)
  \,.
\end{align*}
One therefore finds that the any critical circle must have radius
\begin{equation}
\label{EQrho1}
\rho = \sqrt{\frac{\sqrt{(2 + 2\lambda + |f|^2)^2 +12}-(2 + 2\lambda + |f|^2)}{6}}\,,
\end{equation}
which does indeed satisfy $\BH(S_\rho) = 0$.
Finally let us consider the case where $\L$ is a rotation of the plane $\text{span}\{e_1,e_2\}$.
If the rotation is not about the origin, then the invariance condition implies $\theta = 2k\pi$ for $k\in\Z$;
that is, the rotation is trivial, and one obtains the circle with radius given by \eqref{EQrho1} as the only
possible solution.
Finally, if the rotation is about the origin, a necessary condition for $S_\rho$ to be critical is
\begin{align}
0 = \IP{\BH(S_\rho)}{\k}
  &=
  - 2\rho^{-2}\cos\theta
  + \frac12\rho^{-4}
  - \lambda\rho^{-2}
  + \rho^2\IP{\L\k}{\k}|\k|^2
  - f\IP{\k}{\L\tau}
\notag\\&\quad
  - \frac12\rho^4|\L\k|^2|\k|^2
  - \frac12|f|^2|\k|^2
  + \IP{-\rho^2\L\k}{\L\k}
  + f\IP{\tau}{\L\k}
\notag\\&=
    \frac12\rho^{-4}
  + \Big(
   - \cos\theta
   - \lambda
   - \frac12|f|^2
    \Big)\rho^{-2}
  - \frac32
\notag\\\notag&\quad
  + f\rho^{-1}\Big( \cos\big(\text{$\textstyle \frac\pi2$} + \theta)
                  - \cos(\text{$\textstyle \frac\pi2$}-\theta)\Big)\,,\intertext{that is}
0 &=
    \frac12\rho^{-4}
  + \Big(
   - \cos\theta
   - \lambda
   - \frac12|f|^2
    \Big)\rho^{-2}
  - 2f\rho^{-1}\sin\theta
  - \frac32
\,.
\label{EQpolyroot}
\end{align}
Therefore if $S_\rho$ is to be critical for $\SH$ its radius must be a real positive root of the polynomial
\eqref{EQpolyroot}.
\end{proof}


\bibliographystyle{plain}
\bibliography{Wh12}

\end{document}